\documentclass[oneside,english]{amsart}
\usepackage[T1]{fontenc}
\usepackage[latin9]{inputenc}
\usepackage{geometry}
\usepackage{color}
\usepackage{babel}
\usepackage{verbatim}
\usepackage{units}
\usepackage{amsbsy}
\usepackage{amstext}
\usepackage{amsthm}
\usepackage{amssymb}
\usepackage{setspace}
\usepackage[unicode=true,pdfusetitle,
 bookmarks=true,bookmarksnumbered=true,bookmarksopen=true,bookmarksopenlevel=1,
 breaklinks=false,pdfborder={0 0 1},backref=section,colorlinks=true]
 {hyperref}

\makeatletter
\numberwithin{equation}{section}
\numberwithin{figure}{section}
\theoremstyle{plain}
\newtheorem{thm}{\protect\theoremname}[section]
\theoremstyle{definition}
\newtheorem{defn}[thm]{\protect\definitionname}
\theoremstyle{plain}
\newtheorem{prop}[thm]{\protect\propositionname}
\theoremstyle{remark}
\newtheorem{rem}[thm]{\protect\remarkname}
\theoremstyle{plain}
\newtheorem{cor}[thm]{\protect\corollaryname}
\theoremstyle{plain}
\newtheorem{lem}[thm]{\protect\lemmaname}

\usepackage{tikz}

\makeatother

\providecommand{\corollaryname}{Corollary}
\providecommand{\definitionname}{Definition}
\providecommand{\lemmaname}{Lemma}
\providecommand{\propositionname}{Proposition}
\providecommand{\remarkname}{Remark}
\providecommand{\theoremname}{Theorem}

\begin{document}
\title{$L^{p}$-Expander Graphs}
\author{Amitay Kamber}
\thanks{Institute of Mathematics, The Hebrew University of Jerusalem, amitay.kamber@mail.huji.ac.il}
\begin{abstract}
We discuss how graph expansion is related to the behavior of $L^{p}$-functions
on the covering tree. We show that the non-trivial eigenvalues of
the adjacency operator on a $(q+1)$-regular graph are bounded by
$q^{\nicefrac{1}{p}}+q^{\nicefrac{\left(p-1\right)}{p}}$ -- the
$L^{p}$-norm of the operator on the covering tree -- if and only
if properly averaged lifts of functions from the graph to the tree
lie in $L^{p+\epsilon}$ for every $\epsilon>0$. We generalize the
result to operators on edges and to bipartite graphs.

The work is based on a combinatorial interpretation of representation-theoretic
ideas.
\end{abstract}

\maketitle

\section{Introduction}

The goal of this paper is to put on record some claims about expander
graphs, which characterize them by the properties of $L^{p}$-functions
on their infinite covering tree. Let us first set some notations.

Fix $2\le q\in\mathbb{N}$. Let $X$ be a finite, connected, $\left(q+1\right)$-regular
graph. Let $T$ be the $\left(q+1\right)$-regular tree which is the
universal cover of $X$ and let $\pi\colon T\rightarrow X$ be a covering
map. Let $v_{0}\in V_{T}$ be a vertex in $T$. For every function
$f\colon V_{X}\rightarrow\mathbb{C}$ (i.e.~a function defined on
the vertices of $X$) let $\tilde{f}\colon V_{T}\rightarrow\mathbb{C}$
be the lift of $f$ to $T$, i.e.~$\tilde{f}=f\circ\pi$. Let $\rho_{v_{0}}(\tilde{f})(v)=\frac{1}{q_{v}}\sum_{v'}\tilde{f}(v')$
when $v'$ runs over the sphere in $T$ of radius $d=dist(v_{0},v)$
around $v_{0}$ and $q_{v}$ is the size of the sphere (i.e.~$q_{v_{0}}=1$
and for $v\ne v_{0}$, $q_{v}=\left(q+1\right)q^{d(v,v_{0})-1}$).
Let $L_{0}^{2}(X)=\left\{ f\colon V_{X}\rightarrow\mathbb{C}:\,\sum_{v\in V_{X}}f(v)=0\right\} $
be the subspace orthogonal to the constant function. If $X$ is not
bipartite, let $L_{00}^{2}(V_{X})=L_{0}^{2}(V_{X})$, otherwise let
$L_{00}^{2}(V_{X})$ be the subspace of $f\in L_{0}^{2}(V_{X})$ for
which the sum of $f$ is $0$ over each of the sides of $X$. We wish
to understand the behavior at infinity of $\rho_{v_{0}}(\tilde{f})$
for $\tilde{f}\in L_{00}^{2}(V_{X})$. The function $\tilde{f}$ itself
is periodic, and therefore if $f\ne0$ it is not in $L^{p}(V_{T})$
for any $p<\infty$.

Let $A\colon L^{2}(V_{X})\rightarrow L^{2}(V_{X})$ be the vertex
adjacency operator, $Af(v)=\sum_{v'\sim v}f(v')$. The subspace $L_{00}^{2}(V_{X})$
is the space of functions orthogonal to the eigenvectors with trivial
eigenvalues $\pm\left(q+1\right)$ of $A$. Let $\lambda(X)$ be the
largest absolute value of an eigenvalue of $A$ on $L_{00}^{2}(V_{X})$.
It is standard that $\lambda(X)<q+1$ (see Subsection \ref{subsec:Action-on-Finite-graphs})
and that the graph $X$ is a good expander if $\lambda(X)$ is small
(see \cite{hoory2006expander}). In particular, $X$ is called Ramanujan
if $\lambda(X)\le2\sqrt{q}$. By the Alon-Boppana theorem (\cite{nilli1991second})
Ramanujan graphs are considered the best expanders.

Notice that the function $\hat{\lambda}(p)=q^{\nicefrac{1}{p}}+q^{\nicefrac{\left(p-1\right)}{p}}$
for $p\in[2,\infty]$ is increasing with $\hat{\lambda}(2)=2\sqrt{q}$
and $\hat{\lambda}(\infty)=q+1$.
\begin{thm}
\label{thm:Basic graph claim}For $p\ge2$, $\lambda(X)\le q^{\nicefrac{1}{p}}+q^{\nicefrac{\left(p-1\right)}{p}}$
if and only if for every $f\in L_{00}^{2}(V_{X})$ and $v_{0}\in T$,
$\rho_{v_{0}}\left(\tilde{f}\right)\in L^{p+\epsilon}(V_{T}$) for
every $\epsilon>0$.

In particular, $X$ is a Ramanujan graph (i.e.~$\lambda(X)\le2\sqrt{q}$)
if and only if for every $f\in L_{00}^{2}(V_{X})$ and $v_{0}\in T$,
$\rho_{v_{0}}\left(\tilde{f}\right)\in L^{2+\epsilon}(V_{T})$ for
every $\epsilon>0$.
\end{thm}

One of the motivations for the definition of a Ramanujan graph is
the classical result of Kesten (\cite{kesten1959symmetric}), stating
that the norm of $A$ on $L^{2}(V_{T})$ is $2\sqrt{q}$. This is
the case $p=2$ of the following theorem:
\begin{thm}
\label{thm:Lp bound}The value $\hat{\lambda}(p)=q^{\nicefrac{1}{p}}+q^{\nicefrac{\left(p-1\right)}{p}}$
is the norm of the adjacency operator $A$ on $L^{p}(V_{T})$, $1\le p\le\infty$.

More precisely, the spectrum of $A$ on $L^{p}(V_{T})$, $1\le p\le\infty$
is $\left\{ \theta+q\theta^{-1}:\theta\in\mathbb{C},\,q^{\nicefrac{1}{p'}}\le\left|\theta\right|\le q^{\nicefrac{\left(p'-1\right)}{p'}}\right\} $,
for $p'=\max\left\{ p,\nicefrac{p}{\left(p-1\right)}\right\} $.
\end{thm}

Theorem \ref{thm:Basic graph claim} allows us to define:
\begin{defn}
For $p\ge2$, a graph is an \emph{$L^{p}$-expander} if it satisfies
one of the equivalent conditions of Theorem \ref{thm:Basic graph claim}.
\end{defn}

In particular, our definition of an $L^{2}$-expander graph is the
same as a Ramanujan graph. The referee suggested using the term an
$L^{p}$-Ramanujan graph instead of an $L^{p}$-expander, which is
perhaps a better choice since the definition is based on comparing
the graph to the universal cover $T$. We preferred this notion since
$L^{p}$-expanders are simply expanders with an explicit spectral
gap.

Theorem \ref{thm:Basic graph claim} and Theorem \ref{thm:Lp bound}
are part of a larger theory which we develop in this paper, which
include operators acting also on functions on the directed edges $E_{X}$
of $X$ (i.e.~each non-oriented edge of $X$ is counted twice in
$E_{X}$). Let us state here the main result without giving all the
definitions. See Theorem \ref{thm:Full Theorem} for a precise and
extended form.
\begin{thm}
\label{thm:Full Theorem-1}Let $X$ be a finite, connected, non-bipartite,
$\left(q+1\right)$-regular graph. For $p\ge2$, the following are
equivalent:
\begin{enumerate}
\item Every eigenvalue $\lambda$ of $A$ on $L_{0}^{2}(V_{X})$ satisfies
$\left|\lambda\right|\le q^{\nicefrac{1}{p}}+q^{\nicefrac{\left(p-1\right)}{p}}$.
\item For every $f\in L_{0}^{2}(V_{X})$ and $v_{0}\in T$, $\rho_{v_{0}}\left(\tilde{f}\right)\in L^{p+\epsilon}(V_{T})$
for every $\epsilon>0$.
\item Every eigenvalue $\lambda$ of Hashimoto's non-backtracking operator
$h_{NB}$ on $L_{0}^{2}(E_{X})$ satisfies $\left|\lambda\right|\le q^{\nicefrac{\left(p-1\right)}{p}}$.
\item The eigenvalues of every ``Hecke operator'' $h$ on $L_{0}^{2}(V_{X})\oplus L_{0}^{2}(E_{X})$
are contained in the spectrum of $h$ on $L^{p}(V_{T})\oplus L^{p}(E_{X})$. 
\item Let $A_{k}\colon L^{2}\left(V_{X}\right)\to L^{2}\left(V_{X}\right)$
be the ``non-backtracking-distance-$k$ operator''. For every $k\ge0$,
every eigenvalue $\lambda_{k}$ of $A_{k}$ on $L_{0}^{2}(V_{X})$
satisfies $\left|\lambda_{k}\right|\le(k+1)q^{k\nicefrac{\left(p-1\right)}{p}}$.
\end{enumerate}
\end{thm}

The theory was developed as part of a generalization to higher dimensions,
which appears in \cite{kamber2016lpcomplex}. In this paper we discuss
its extension to finite, connected, bipartite, $\left(q_{0}+1,q_{1}+1\right)$-biregular
graphs, where $q_{0}<q_{1}$. The graph $X$ has two types of vertices -- $0$
and $1$, and each vertex of type $i\in\left\{ 0,1\right\} $ is contained
in $q_{i}+1$ edges. The non-precise form of the theory in this case
is summarized in the following theorem. See Subsection \ref{subsec:Bipartite-Biregular-Graphs}
for full details.
\begin{thm}
\label{thm:Bipartite theorem}Let $X$ be a finite, connected, bipartite,
\textup{$\left(q_{0}+1,q_{1}+1\right)$-biregular} graph. For $p\ge2$,
the following are equivalent:
\begin{enumerate}
\item For every $f\in L_{00}^{2}(V_{X})$ and $v_{0}\in T$, $\rho_{v_{0}}\left(\tilde{f}\right)\in L^{p+\epsilon}(V_{T}$)
for every $\epsilon>0$.
\item Every eigenvalue $\lambda$ of Hashimoto's non-backtracking operator
$\tilde{h}_{NB}$ on $L_{0}^{2}(\tilde{E}_{X})$ satisfies $\left|\lambda\right|\le\left(q_{0}q_{1}\right)^{\nicefrac{\left(p-1\right)}{p}}$.
\item The eigenvalues of every ``Hecke operator'' $h$ on $L_{00}^{2}(V_{X})\oplus L_{0}^{2}(\tilde{E}_{X})$
are contained in the spectrum of $h$ on $L^{p}(V_{T})\oplus L^{p}(\tilde{E}_{T})$.
\item Let $A_{k}\colon L^{2}\left(V_{X}\right)\to L^{2}\left(V_{X}\right)$
be the ``non-backtracking-distance-$k$ operator'' on vertices.
For every $k\ge0$, every eigenvalue $\lambda_{k}$ of $A_{k}$ on
$L_{00}^{2}(V_{X})$ satisfies $\left|\lambda_{k}\right|\le(k+1)q_{1}\left(q_{0}q_{1}\right)^{\left(\nicefrac{k}{2}\right)\nicefrac{\left(p-1\right)}{p}}$.
\end{enumerate}
For $p=2$, the conditions are equivalent to:
\end{thm}

\begin{itemize}
\item Every eigenvalues $\lambda$ of $A$ on $L_{00}^{2}(V_{X})$ satisfies
$\lambda=0$ or $\sqrt{q_{1}}-\sqrt{q_{0}}\le\left|\lambda\right|\le\sqrt{q_{1}}+\sqrt{q_{0}}$,
and the multiplicity of $\lambda=0$ is $\left|V_{X}^{0}\right|-\left|V_{X}^{1}\right|$,
where $V_{X}^{i}$ is the number of vertices of $X$ of type $i$.
\end{itemize}

\subsection*{Related Results}

Theorem \ref{thm:Full Theorem-1} is connected to some well known
results, mostly about the Zeta function of the graph or the non-backtracking
operator. The condition on the eigenvalues of $A$ for $p=2$ in Theorem
\ref{thm:Bipartite theorem} appears (as a definition) in the work
of Hashimoto (\cite[3.21]{hashimoto1989zeta}), based on similar considerations.
The main new idea of this work is the comparison to $L^{p}$-functions
on the tree, with an extension to $p>2$.

The proof of Alon's second eigenvalue conjecture (\cite{friedman2003proof,bordenave2015new})
is closely related to our analysis. Theorem 20 in \cite{bordenave2015new}
in conjugation with Theorem \ref{thm:Bipartite theorem} show that
a random cover $X'$ of a biregular graph $X$ is an $L^{2+\epsilon}$-expander-cover.
That is, we have a natural decomposition $L^{2}\left(\tilde{E}_{X'}\oplus V_{X'}\right)\cong L^{2}\left(\tilde{E}_{X}\oplus V_{X}\right)\oplus W_{new}$
and $W_{new}$ satisfies the conditions of Theorem \ref{thm:Bipartite theorem}
for $p=2+\epsilon$. In contrast, the covers built in \cite{marcus2013interlacing}
or \cite{hall2018ramanujan}, are not necessarily $L^{2}$-expander-covers.
The construction only promises that the new eigenvalues of $A$ are
bounded from above by $\sqrt{q_{1}}+\sqrt{q_{0}}$ (see also \cite[Question 6.3]{hall2018ramanujan}).

One application of the $L^{p}$-theory developed here is to generalize
the results of \cite{lubetzky2016cutoff,sardari1diameter}, which
concern ``almost-diameter''. The results state that for every $\epsilon>0$
and $\delta>0$ there exists $N$ such that if $X$ is a $\left(q+1\right)$-regular
Ramanujan graph $X$ with $\left|V_{X}\right|>N$, then for every
$x\in V_{X}$, all but $\delta\left|V_{X}\right|$ of the vertices
$y\in V_{X}$ are of distance within $\left[\left(1-\epsilon\right)\log_{q}\left|V_{X}\right|,\left(1+\epsilon\right)\log_{q}\left|V_{X}\right|\right]$
from $x$, which is an optimal result. One can show that for Cayley
expander graphs it is actually enough to assume far weaker conditions,
namely that for every $p>2$ and $\epsilon'>0$, the number of bad
eigenvalues $\lambda$ satisfying $\left|\lambda\right|\ge q^{\nicefrac{1}{p}}+q^{\nicefrac{\left(p-1\right)}{p}}$
is bounded by $C_{\epsilon'}\left|V_{X}\right|^{\nicefrac{2}{p}+\epsilon'}$,
$C_{\epsilon'}$ some constant. Those results will be given elsewhere,
but see \cite{golubev2019cutoff} for analogous results for hyperbolic
surfaces.

\subsection*{Arbitrary Graphs}

The results of this paper do not extend naturally to a general graph
$X$, i.e.~which is not regular nor biregular. In particular, we
do not know the relations between the eigenvalues of the adjacency
operator $A$, the ``non-backtracking-distance-$k$ operator'' $A_{k}$
and the non-backtracking operator $h_{NB}$. 

One can, however, give some $L^{p}$-bounds on operators of the covering
tree $T$ of $X$. For example, up to $O(k)$, the norm of $A_{k}$
on $L^{p}\left(V_{T}\right)$ is bounded by the $\nicefrac{\left(p-1\right)}{p}$-th
power of $N_{k}$ -- the maximal number of vertices on a sphere of
radios $k$. We have that $N_{k}^{1/k}\rightarrow\text{gr}(T)$, the
growth rate of the tree (see \cite{angel2015non}). Similarly, the
spectrum of the non-backtracking operator $h_{NB}$ on $L^{p}(E_{X})$
is bounded in absolute value by $\text{gr}(T)^{\nicefrac{\left(p-1\right)}{p}}$.
See \cite{angel2015non} for similar calculations for $p=2$.

\subsection*{Structure of the article}

The work is divided into two sections -- Section \ref{sec:Operators-on-Vertices}
concerns operators acting on maps on vertices and Section \ref{sec:Operators-on-Edges}
concerns operators acting on maps on directed edges. 

Section \ref{sec:Operators-on-Vertices} contains the proofs of Theorem
\ref{thm:Basic graph claim} and Theorem \ref{thm:Lp bound}. Theorem
\ref{thm:Basic graph claim} can be proved directly. We will take
a slightly longer path, introducing along the way the basic notions
of the vertex Hecke algebra $H_{0}$ and the Satake isomorphism. We
define the algebra in Subsection \ref{subsec:The-Vertex-Hecke-Algebra-of-a-Tree}.
We study the Satake isomorphism and the irreducible representations
of $H_{0}$ in Subsection \ref{subsec:Representations-of_H_0}. In
Subsection \ref{subsec:Geometric-Realization} we show that each irreducible
representation can be realized on functions of the vertices $V_{T}$
of the tree $T$. In Subsection \ref{subsec:Action-on-Finite-graphs}
we prove Theorem \ref{thm:Basic graph claim} and finally in Subsection
\ref{subsec:L_p-Spectrum-of-Hecke} we prove Theorem \ref{thm:Lp bound}.

The second section of this paper is devoted to a generalization of
the theory from maps on vertices to maps on directed edges. While
this generalization is interesting for its own right, its advantage
is apparent in \cite{kamber2016lpcomplex}, where we study high dimensional
$L^{p}$-expanders (see \cite{lubotzky2014ramanujan} for an introduction
to the subject of high dimensional expanders). However, since this
algebra is more complicated, and in particular not commutative, working
with it requires more preliminaries. In Subsection \ref{subsec:The-Directed-Edge-Hecke}
we define the Iwahori-Hecke algebra $H_{\phi}$ of the tree, acting
on functions on directed edges. In Subsection \ref{subsec:Representation-Theory-of}
we study the representations of $H_{\phi}$. In Subsection \ref{subsec:Geometric-Realization-1}
we study the realizations of the representations on functions of the
directed edges $E_{T}$ of the tree. In Subsection \ref{subsec:L_p-Expander-Theorem}
we present Theorem \ref{thm:Full Theorem} which combines all the
results about regular graphs. Finally, in Subsection \ref{subsec:Bipartite-Biregular-Graphs}
we study the corresponding theory for bipartite biregular (but not
regular) graphs. 

Our analysis is based on constructions from the representation theory
of $p$-adic Lie groups, although no prior knowledge of it is assumed.
The main contribution of this work is the interpretation of the representation-theoretic
statements into simple combinatorial language, as well as an extension
of some results to the case $p>2$.

\section*{Preliminaries from Representation Theory}

Let us collect here some basic definitions and facts from representation
theory. Note that the first section of this work concerns the commutative
vertex Hecke algebra $H_{0}$ whose analysis does not require most
of those facts. The reader is referred to \cite{etingof2011introduction}
for basic representation theory.

An algebra $H$ in this work is a vector space over $\mathbb{C}$,
with multiplication satisfying the usual properties (including associativity),
and a unit $Id=Id_{H}$. Every subset of elements $S\subset H$ generates
a subalgebra $H_{S}$, which is the intersection of the subalgebras
(with the same unit $Id_{H'}=Id_{H}$) $H'\subset H$ containing $S$.
We say that the algebra $H$ is generated by $S$ if $H_{S}=H$. We
say that the algebra $H$ is freely generated by an element $A\in H$
if there is an algebra isomorphism $f\colon\mathbb{C}\left[x\right]\to H$
defined by $f(x)=A$. An algebra representation $\left(\pi,V\right)$
of $H$ is an algebra homomorphism $\pi\colon H\to\text{End}_{\mathbb{C}}\left(V\right)$
such that $\pi(Id_{H})=Id_{V}$. To make the notations simpler we
sometimes omit $\pi$ and refer to $V$ directly as the representation.

A homomorphism of two $H$-representations $\left(\pi_{1},V_{1}\right)$
and $\left(\pi_{2},V_{2}\right)$ is a linear map $\varphi\colon V_{1}\to V_{2}$
such that for all $v\in V_{1}$, $h\in H$, it holds that $\pi_{2}\left(h\right)\varphi\left(v\right)=\varphi\left(\pi_{1}\left(h\right)v\right)$.

If $U\subset V$ is a $\pi\left(H\right)$-invariant subspace of $V$
we say that $U$ is a subrepresentation of $\left(\pi,V\right)$.
If $U$ is a subrepresentation of $V$ there is a natural action of
$H$ on the vector space $V/U$, to which we call a quotient representations
of $\left(\pi,V\right)$. If $\varphi\colon V_{1}\to V_{2}$ is an
$H$-algebra homomorphism then $\ker\varphi\subset V_{1}$ is a subrepresentation
of $V_{1}$ and $\varphi\left(V_{1}\right)\subset V_{2}$ is a subrepresentation
of $V_{2}$ isomorphic to the quotient $V_{1}/\ker\varphi$. In particular,
if $\varphi$ is onto then $V_{2}$ is isomorphic to a quotient of
$V_{1}$.

If $U_{1},U_{2},..,U_{k}$ are subrepresentation of $V$ and $V=\oplus_{i=1}^{k}U_{i}$
we say that the representation $V$ is a direct sum of the representations
$U_{1},U_{2},..,U_{k}$. An algebra representation is called indecomposable
if it is not a direct sum of two proper non-trivial subrepresentations.
It is called irreducible if there is no proper non-trivial subrepresentations. 

An $H$-representation $\left(\pi,V\right)$ is also called a left
$H$-module. Similarly, one may define a right $H$-module $\left(\rho,W\right)$
by a linear transformation $\rho\colon H\to\text{\text{End}}_{\mathbb{C}}\left(W\right)$
such that $\rho\left(hh'\right)=\rho\left(h'\right)\rho\left(h\right)$.
Given a left $H$-module $\left(\pi,V\right)$ and a right $H$-module
$\left(\rho,W\right)$ one may define the vector space $W\otimes_{H}V$,
which is the quotient of the vector space $V\otimes W$ by the vector
subspace spanned by $\left\{ \rho\left(h\right)w\otimes v-w\otimes\pi\left(h\right)v=0:v\in V,w\in W,h\in H\right\} $.
Moreover, if $W$ is an $H'$-representation then $W\otimes_{H}V$
is also an $H'$-representation.

We call an algebra $H$ a $\ast$-algebra if it has an involution
$\ast\colon H\to H$, i.e.~a map satisfying for $h_{1},h_{2}\in H$,
$\alpha\in\mathbb{C}$, $\left(h_{1}+\alpha h_{2}\right)^{\ast}=h_{1}^{\ast}+\overline{\alpha}h_{2}^{\ast}$
and $\left(h_{1}h_{2}\right)^{\ast}=h_{2}^{\ast}h_{1}^{\ast}$. A
representation $\left(\pi,V\right)$ of a $\ast$-algebra $H$ is
called unitary if there exists an inner product $\left\langle \cdot,\cdot\right\rangle $
on $V$ satisfying $\left\langle \pi\left(h\right)v_{1},v_{2}\right\rangle =\left\langle v_{1},\pi\left(h^{*}\right)v_{2}\right\rangle $
for every $v_{1},v_{2}\in V$ and $h\in H$. We will need the standard
claim:
\begin{prop}
\label{prop:finite dimensional unitary decomposes}Every finite dimensional
unitary representation $\left(\pi,V\right)$ of a $\ast$-algebra
$H$ decomposes into a direct sum of irreducible representations.
\end{prop}

\begin{proof}
Assume $\{0\}\ne V'\subset V$ is a proper subrepresentation. Let
$U=\{u\in V:\forall v\in V'\,\left\langle v,u\right\rangle =0\}$.
Since $\left\langle \cdot,\cdot\right\rangle $ is an inner product
we have $V=V'\oplus U$ as vector spaces. Moreover if $u\in U$, $h\in H$
then for every $v\in V'$, $\left\langle v,\pi\left(h\right)u\right\rangle =\left\langle \pi\left(h^{\ast}\right)v,u\right\rangle =0$.
Therefore $\pi\left(h\right)u\in U$ and $U$ is also a subrepresentation.
The claim follows by induction.
\end{proof}

\section{\label{sec:Operators-on-Vertices}Operators on Vertices}

\subsection{\label{subsec:The-Vertex-Hecke-Algebra-of-a-Tree}The Vertex Hecke
Algebra of a Regular Tree}

Let $T$ be the $\left(q+1\right)$-regular tree and let $V_{T}$
be its set of vertices. Let $d\colon V_{T}\times V_{T}\rightarrow\{0,1,...\}$
be the natural distance function. 
\begin{defn}
Let $A_{k}\colon\mathbb{C}^{V_{T}}\rightarrow\mathbb{C}^{V_{T}}$,
$k=0,1,\dots$ be the operator: 
\[
A_{k}f(v)=\sum_{v':d(v,v')=k}f(v').
\]
\end{defn}

Notice that $A_{0}=Id$, that $A_{1}=A$ is the vertex adjacency operator
of $T$, and that for $k\ge1$, $A_{k}$ sums $\left(q+1\right)q^{k-1}\approx q^{k}$
different vertices. The \emph{Hecke relations} can be easily verified:
\emph{
\begin{align*}
A^{2} & =A_{2}+\left(q+1\right)A_{0}\\
AA_{k} & =A_{k+1}+qA_{k-1}\,\,\,\,\,\text{for }k=2,3,\dots
\end{align*}
}
\begin{defn}
The \emph{vertex Hecke algebra} $H_{0}$ (sometimes called the \emph{spherical
Hecke algebra}), is the algebra spanned as a vector space by $A_{k}$,
$k\ge0$.
\end{defn}

By the Hecke relations $H_{0}$ is indeed an algebra. The relations
also show that $H_{0}$ is commutative and freely generated by $A=A_{1}$.

There is a more abstract definition of the vertex Hecke algebra. Define:
\begin{defn}
Let $S$ be a discrete set. We say that a linear operator $h\colon\mathbb{C}^{S}\rightarrow\mathbb{C}^{S}$
is \emph{row and column finite} if it can be written as $hf(x)=\sum_{y\in S}\alpha_{x,y}f(y)$,
for some $\alpha\colon S\times S\rightarrow\mathbb{C}$, with $\#\left\{ y:\alpha_{x,y}\ne0\right\} <\infty$
and $\#\left\{ y:\alpha_{y,x}\ne0\right\} <\infty$ for every $x\in S$. 
\end{defn}

Notice that every operator $h\in H_{0}$ is row and column finite
since this is true for the spanning vectors $A_{k}$, $k\ge0$. 
\begin{prop}
\label{prop:Hecke_commutes}Let $\gamma\in\mbox{Aut}(T)$ be an automorphism
of the tree. Then $\gamma$ acts naturally on $\mathbb{C}^{V_{T}}$
by $\gamma\cdot f(x)=f(\gamma^{-1}x)$. Let $h\colon\mathbb{C}^{V_{T}}\rightarrow\mathbb{C}^{V_{T}}$
a linear operator. Then $h\in H_{0}$ if and only if $h$ is row and
column finite and the action of $h$ on $\mathbb{C}^{V_{T}}$ commutes
with the action of each $\gamma\in\mbox{Aut}(T)$ on $\mathbb{C}^{V_{T}}$.
\end{prop}

\begin{proof}
Since automorphisms preserve distances in $T$, the only if part follows.

As for the if part, write $h\colon\mathbb{C}^{V_{T}}\rightarrow\mathbb{C}^{V_{T}}$
as $hf(x)=\sum_{y\in V_{T}}\alpha_{x,y}f(y)$, as in the definition
of a row and column finite operator. Assume that $h$ commutes with
every $\gamma\in\mbox{Aut}(T)$. If $x,y,x',y'\in V_{T}$, $d(x,y)=d(x',y')$,
then there exists $\gamma\in\mbox{Aut}(T)$ such that $\gamma(x')=x$,
$\gamma(y')=y$. Since $h\gamma=\gamma h$ we have $\alpha_{x,y}=\alpha_{x',y'}$.
Therefore $\alpha_{x,y}$ depends only on $d(x,y)$ which means that
$h\in H_{0}$.
\end{proof}
\begin{rem}
The definition of a Hecke algebra as an algebra of operators commuting
with automorphisms appears in a similar context in \cite{first2016ramanujan,kamber2016lpcomplex}.
In both places it is used to extend the definition of Ramanujan graphs
to higher dimensional complexes. 
\end{rem}

\subsection{\label{subsec:Representations-of_H_0}Representations of $H_{0}$}

The following theorem is called the \emph{Satake isomorphism}:
\begin{thm}
The algebra $H_{0}$ is isomorphic to the subalgebra of $\mathbb{C}\left[x,x^{-1}\right]$
which is invariant with respect to the automorphism $x\leftrightarrow x^{-1}$.
The isomorphism is given by:
\begin{align*}
A_{0} & \leftrightarrow1\\
A & \leftrightarrow\hat{A}=q^{\nicefrac{1}{2}}\left(x+x^{-1}\right),
\end{align*}
and for $k\ge2$:
\begin{eqnarray*}
A_{k} & \leftrightarrow & \hat{A}_{k}=q^{\nicefrac{k}{2}}\left(x^{k}+x^{-k}+\left(1-q^{-1}\right)\left(x^{k-2}+x^{k-4}+...+x^{-k+2}\right)\right)\\
 &  & \,\,\,=q^{\nicefrac{\left(k-1\right)}{2}}\left(q^{\nicefrac{1}{2}}x-q^{\nicefrac{1}{2}}x^{-1}\right)^{-1}\left(x^{k-1}\left(qx^{2}-1\right)-x^{-k+1}\left(qx^{-2}-1\right)\right).
\end{eqnarray*}
\end{thm}

\begin{rem}
The algebra $H_{0}$ is freely generated by $A_{1}$ and as such is
isomorphic to $\mathbb{C}\left[x\right]$. The same is true for the
invariant subalgebra of $\mathbb{C}\left[x,x^{-1}\right]$, which
is generated by $x+x^{-1}$. The important statement in the theorem
is the explicit description of the isomorphism.
\end{rem}

\begin{proof}
Since both algebras are freely generated (as algebras) by a single
element, $A\rightarrow q^{\nicefrac{1}{2}}\left(x+x^{-1}\right)$
indeed defines an isomorphism. 

The following calculations verify the explicit description:
\begin{align*}
\hat{A}^{2} & =q\left(x+x^{-1}\right)^{2}=q\left(x^{2}+x^{-2}+2\right)\\
 & =q+1+q\left(x^{2}+x^{-2}+1-q^{-1}\right)\\
 & =\left(q+1\right)\hat{A}_{0}+\hat{A}_{2},
\end{align*}
and for $k\ge2$:
\begin{eqnarray*}
\hat{A}\cdot\hat{A}_{k} & = & q^{\nicefrac{1}{2}}\left(x+x^{-1}\right)\cdot q^{\nicefrac{k}{2}}\left(x^{k}+x^{-k}+(1-q^{-1})(x^{k-2}+x^{k-4}+...+x^{-k+2})\right)\\
 & = & q^{\nicefrac{\left(k+1\right)}{2}}\left(x^{k+1}+x^{-k-1}+x^{k-1}+x^{1-k}+(1-q^{-1})(x^{k-1}+x^{1-k})+(2-2q^{-1})(x^{k-3}+x^{k-5}+...+x^{3-k})\right)\\
 & = & q^{\nicefrac{\left(k+1\right)}{2}}\left(x^{k+1}+x^{-1-k}+(1-q^{-1})(x^{k-1}+x^{k-3}+...+x^{1-k})\right)\\
 &  & +q\cdot q^{\nicefrac{\left(k-1\right)}{2}}\left(x^{k-1}+x^{1-k}+(1-q^{-1})(x^{k-3}+x^{k-5}+...+x^{3-k})\right)\\
 & = & \hat{A}_{k+1}+q\hat{A}_{k-1}.
\end{eqnarray*}

Finally, we have: 
\begin{align*}
\hat{A}_{k} & =q^{\nicefrac{k}{2}}\left(x^{k}+x^{-k}+\left(1-q^{-1}\right)\left(x^{k-2}+x^{k-4}+...+x^{-k+2}\right)\right)\\
 & =q^{\nicefrac{k}{2}}\left(x^{k}+x^{-k}+\left(1-q^{-1}\right)\left(x-x^{-1}\right)^{-1}\left(x^{k-1}-x^{-k+1}\right)\right)\\
 & =q^{\nicefrac{k}{2}}\left(x^{k-1}\left(x+\left(1-q^{-1}\right)\left(x-x^{-1}\right)^{-1}\right)+x^{-k+1}\left(x^{-1}-\left(1-q^{-1}\right)\left(x-x^{-1}\right)^{-1}\right)\right)\\
 & =q^{\nicefrac{k}{2}}\left(x-x^{-1}\right)^{-1}\left(x^{k-1}\left(x^{2}-q^{-1}\right)-x^{-k+1}\left(x^{-2}-q^{-1}\right)\right)\\
 & =q^{\nicefrac{(k-1)}{2}}\left(q^{\nicefrac{1}{2}}x-q^{\nicefrac{1}{2}}x^{-1}\right)^{-1}\left(x^{k-1}\left(qx^{2}-1\right)-x^{-k+1}\left(qx^{-2}-1\right)\right).
\end{align*}
\end{proof}
Let us twist the Satake isomorphism by choosing $\theta=q^{\nicefrac{1}{2}}x$.
Write $\tilde{\theta}=q\theta^{-1}=q^{\nicefrac{1}{2}}x^{-1}$.
\begin{cor}
\label{cor:Twisted-Satake}The algebra $H_{0}$ is isomorphic to the
subalgebra of $\mathbb{C}\left[\theta,\theta^{-1}\right]$ which is
invariant with respect to the automorphism $\theta\leftrightarrow\tilde{\theta}=q\theta^{-1}$.
The isomorphism is given by:
\begin{align*}
A_{0} & \leftrightarrow1\\
A & \leftrightarrow A(\theta)=\theta+\tilde{\theta},
\end{align*}
and for $k\ge1$:
\begin{eqnarray*}
A_{k} & \leftrightarrow & A_{k}(\theta)=\theta^{k}+\tilde{\theta}^{k}+\left(1-q^{-1}\right)\sum_{i=1}^{k-1}\theta^{k-i}\tilde{\theta}^{i}=\\
 &  & =\left(\theta-\tilde{\theta}\right)^{-1}\left(\theta^{k-1}\left(\theta^{2}-1\right)-\tilde{\theta}^{k-1}\left(\tilde{\theta}^{2}-1\right)\right),
\end{eqnarray*}
where the last equality holds for $\theta\ne\tilde{\theta}$.
\end{cor}

We can now classify the irreducible representations of $H_{0}$. For
$0\ne\theta\in\mathbb{C}$, write $A_{k}(\theta)$ as in Corollary
\ref{cor:Twisted-Satake}. 
\begin{cor}
The linear function $\pi_{\theta}^{0}\colon H_{0}\to V_{\theta}\cong\mathbb{C}$
given by $\pi_{\theta}^{0}\left(A_{k}\right)=A_{k}(\theta)$ defines
a representation of $H_{0}$. The representations $\left(\pi_{\theta}^{0},V_{\theta}\right)$
and $\left(\pi_{\theta'}^{0},V_{\theta'}\right)$ are isomorphic if
and only if $\theta'=q\theta^{-1}$ or $\theta'=\theta$. 

Each irreducible finite dimensional representation of $H_{0}$ is
one dimensional and is isomorphic to one of the representations $V_{\theta}$,
for $0\ne\theta\in\mathbb{C}$.
\end{cor}

\begin{proof}
Every eigenvector of $\pi\left(A\right)$ in a representation $\left(\pi,V\right)$
of $H_{0}$ spans a subrepresentation, and therefore each irreducible
finite dimensional representation is one dimensional.

By the (twisted) Satake isomorphism, $\left(\pi_{\theta}^{0},V_{\theta}\right)$
is indeed a representation of $H_{0}$. 

On the other hand, an irreducible representation $\left(\pi,V\right)$
is parameterized by the eigenvalue $\lambda$ of $\pi\left(A\right)$.
Such a representation is isomorphic to $V_{\theta}$ if and only if
$\theta+q\theta^{-1}=\lambda$. This equation always has one or two
solutions $\theta,\theta'\in\mathbb{C}$ satisfying $\theta'=q\theta^{-1}$.
\end{proof}
We say that $\theta$ (or $q\theta^{-1}$) is the \emph{Satake parameter}
of the representation $V_{\theta}$. For a general operator $h\in H_{0}$,
write $h(\theta)\in\mathbb{C}$ for the eigenvalue of $\pi_{\theta}^{0}\left(h\right)$
on $V_{\theta}$.

We will need the following estimates for the representation $V_{\theta}$:
\begin{lem}
\label{lem:Growth Lemma}Let $0\ne\theta\in\mathbb{C}$ satisfy $\left|\theta\right|\ge q\left|\theta\right|^{-1}$.
Then,

(1) For every $k\ge0$, $\left|A_{k}(\theta)\right|\le\left(k+1\right)\left|\theta\right|^{k}$.

(2) There exists an infinite number of $k$-s for which $\left|A_{k}(\theta)\right|\ge2^{-3}\cdot\left|\theta\right|^{k}$.
\end{lem}

\begin{proof}
For (1), 
\begin{align*}
\left|A_{k}(\theta)\right| & =\left|\theta^{k}+\tilde{\theta}^{k}+\left(1-q^{-1}\right)\sum_{i=1}^{k-1}\theta^{k-i}\tilde{\theta}^{i}\right|\le\left|\theta\right|^{k}+\left|\tilde{\theta}\right|^{k}+\left(1-q^{-1}\right)\sum_{i=1}^{k-1}\left|\theta^{k-i}\tilde{\theta}^{i}\right|\\
 & \le\left|\theta\right|^{k}+\left|\theta\right|^{k}+\left(1-q^{-1}\right)\sum_{i=1}^{k-1}\left|\theta\right|^{k}\le\left(k+1\right)\left|\theta\right|^{k}.
\end{align*}

For (2) consider the following 3 cases. 

(a) If $\theta=\tilde{\theta}$ then $\theta=\tilde{\theta}=\pm\sqrt{q}$.
Then for every $k$ even, $A_{k}(\theta)$ is a sum of positive terms
and $A_{k}(\theta)\ge\left|\theta\right|^{k}$.

(b) If $\theta\ne\tilde{\theta}$ and $\left|\theta\right|>\left|\tilde{\theta}\right|$
then for $k$ large enough $\left|\tilde{\theta}\right|^{k-1}\left|\tilde{\theta}^{2}-1\right|\le2^{-1}\left|\theta\right|^{k-1}\left|\theta^{2}-1\right|$,
so 
\[
\left|A_{k}(\theta)\right|\ge2^{-1}\left|\theta-\tilde{\theta}\right|^{-1}\left|\theta^{2}-1\right|\left|\theta\right|^{k-1}.
\]

We have that $\left|\theta-\tilde{\theta}\right|\le2\left|\theta\right|$
so $\left|\theta-\tilde{\theta}\right|^{-1}\ge2^{-1}\left|\theta\right|^{-1}$.
Since $\left|\theta\right|^{2}\ge q\ge2$, $\left|\theta^{2}-1\right|\ge2^{-1}\left|\theta\right|^{2}$.
Therefore $\left|A_{k}(\theta)\right|\ge2^{-3}\left|\theta\right|^{k}$.

(c) If $\theta\ne\tilde{\theta}$ and $\left|\theta\right|=\left|\tilde{\theta}\right|$
then $\tilde{\theta}=\bar{\theta}$ and $\left|\theta\right|=q^{\nicefrac{1}{2}}$.
Since $q^{-\nicefrac{1}{2}}\theta$ is on the unit circle and is different
from $\pm1$, there exists an infinite number of $k$-s such that
the imaginary part of $\left(\theta^{2}-1\right)\theta^{k-1}$ is
at least half of its absolute value. For such $k$-s 
\[
\left|\left(\theta^{2}-1\right)\theta^{k-1}-\left(\tilde{\theta}^{2}-1\right)\tilde{\theta}^{k-1}\right|=\left|\left(\theta^{2}-1\right)\theta^{k-1}-\overline{\left(\theta^{2}-1\right)\theta^{k-1}}\right|\ge2^{-1}\left|\left(\theta^{2}-1\right)\theta^{k-1}\right|.
\]
Therefore 
\[
\left|A_{k}(\theta)\right|\ge2^{-1}\left|\theta-\tilde{\theta}\right|^{-1}\left|\theta^{2}-1\right|\left|\theta\right|^{k-1}\ge2^{-3}\left|\theta\right|^{k}.
\]

The last inequality follows from the same computation as in (b).

Combining (a), (b) and (c) gives the explicit constant.
\end{proof}
\begin{rem}
In number theory texts, it is more common to work with $B_{k}=\sum_{0\le l\le\nicefrac{k}{2}}A_{k-2l}$
(see, for example, \cite[4.10]{lubotzky1988ramanujan}). The bounds
for $\left|\theta\right|=\sqrt{q}$ are well known: the bound $\left|B_{k}\left(\theta\right)\right|\le\left(k+1\right)q^{\nicefrac{\left(k-1\right)}{2}}$
with a similar derivation appears in Ramanujan's original conjecture
about the $\tau$ function (\cite[Section 18]{ramanujan1916certain}).
\end{rem}

\subsection{\label{subsec:Geometric-Realization}Geometric Realization}

The construction of $V_{\theta}$ can be realized as a subrepresentation
of the action of $H_{0}$ on $\mathbb{C}^{V_{T}}$ in two ways: the
\emph{sectorial model} and the \emph{spherical model}.

To describe the sectorial model, fix an infinite ray (i.e.~an infinite
non-backtracking path) $R=(v_{0},v_{1},...,)$ on the tree. We define
the \emph{relative distance} $c(v)=c_{R}(v)\in\mathbb{Z}$ of a vertex
$v$ to the ray $R$ as follows: let $c(v_{k})=-k$ for every vertex
$v_{k}$ on the ray $R$, and for any other vertex $v$, if $v_{k}$
is the closest vertex to $v$ among all vertices in $R$, define $c(v)=c(v_{k})+d(v,v_{k})$
- see Figure \ref{fig:Sectorial model}.

\begin{figure}
\begin{tikzpicture} 
\tikzstyle{every node}=[draw,shape=circle]; 
\node[draw=none,fill=none] (c3) at (-2,2) {$c(v)=-4$};
\node[draw=none,fill=none] (c3) at (-2,0) {$c(v)=-3$};
\node[draw=none,fill=none] (c2) at (-2,-2) {$c(v)=-2$};
\node[draw=none,fill=none] (c1) at (-2,-4) {$c(v)=-1$};
\node[draw=none,fill=none] (c0) at (-2,-6) {$c(v)=0$};
\node[draw=none,fill=none] (cm) at (-2,-8) {$c(v)=1$};
\node (v3) at (0,0)  {$v_3$}; 
\node (v2) at (0,-2) {$v_2$}; 
\node (v1) at (0,-4) {$v_1$}; 
\node (v0) at (0,-6) {$v_0$}; 
\node (v3a)  at (2,0) {$\ \ $};
\node (v2a)  at (2,-2) {$\ \ $};
\node (v2aa) at (4,-2) {$\ \ $};
\node (v2ab) at (5,-2) {$\ \ $};
\node (v1a)  at (2,-4) {$\ \ $};
\node (v1aa) at (4,-4) {$\ \ $};
\node (v1ab) at (5,-4) {$\ \ $};
\node (v0a)  at (2,-6) {$\ \ $};
\node (v0aa) at (4,-6) {$\ \ $};
\node (v0ab) at (5,-6) {$\ \ $};

\draw 
(v3) -- (v2)
(v3) -- (v2a)
(v3a) -- (v2aa)
(v3a) -- (v2ab)

(v2) -- (v1)
(v2) -- (v1a)
(v2a) -- (v1aa)
(v2a) -- (v1ab)

(v1) -- (v0)
(v1) -- (v0a)
(v1a) -- (v0aa)
(v1a) -- (v0ab)

;
\draw[dashed] (v3) -- (0,2) 
(v0) -- (0, -8)
(v0) -- (2, -8)
(v3a) -- (0,2)
(v0a) -- (4, -8)

(v0a) -- (5, -8)
(v2aa) -- (6,-4)
(v2aa) -- (6.5,-4)
(v2ab) -- (7,-4)
(v2ab) -- (7.5,-4)

(v1aa) -- (6,-6)
(v1aa) -- (6.5,-6)
(v1ab) -- (7,-6)
(v1ab) -- (7.5,-6)

(v0aa) -- (6,-8)
(v0aa) -- (6.5,-8)
(v0ab) -- (7,-8)
(v0ab) -- (7.5,-8)
;
\end{tikzpicture} 

\caption{\label{fig:Sectorial model}The sectorial distance $c$(v) of a $3$-regular
tree with a chosen ray $\left(v_{0},v_{1},...\right)$. The value
of $c$ is the same for all the vertices in each row of the figure.}
\end{figure}
Define $\tilde{f}_{\theta}\in\mathbb{C}^{V_{T}}$ by $\tilde{f_{\theta}}(v)=\theta{}^{-c(v)}$.
Notice that the relative distance $c$ has the property that every
vertex $v$ has one neighbor $u$ with $c(u)=c(v)-1$, and $q$ neighbors
$u_{1},...,u_{q}$ with $c(u_{i})=c(v)+1$. By this property of $c$,
$\tilde{f}_{\theta}$ is an eigenvector of $A$, with eigenvalue $\theta+q\theta^{-1}$.
Therefore, $\tilde{f}_{\theta}$ spans a representation space of $H_{0}$
isomorphic to $V_{\theta}$, which we call the \emph{sectorial model
of $V_{\theta}$}.

The relative distance to the ray will be used again in Proposition
\ref{prop:Kesten's}, and similar considerations can be used to derive
the (twisted) Satake isomorphism.

While $\tilde{f}_{\theta}$ realizes $V_{\theta}$ as a subrepresentation
of $\mathbb{C}^{V_{T}}$, there exists an infinite number of vertices
$v\in V_{T}$ with $\tilde{f}_{\theta}(v)=1$, so the function is
not in $L^{p}(V_{T})$ for any $p<\infty$. To obtain a representation
with functions of controlled $L^{p}$ norm, look at the vertex $v_{0}\in V_{T}$
that is the start of the ray $R$. Let $f_{\theta}=\rho_{v_{0}}\left(\tilde{f}_{\theta}\right)$,
where $\rho_{v_{0}}\colon\mathbb{C}^{V_{T}}\rightarrow\mathbb{C}^{V_{T}}$
is the spherical average operator of the introduction: for $f\in\mathbb{C}^{V_{T}}$,
$(\rho_{v_{0}}\left(f\right))(v_{0})=f(v_{0})$ and for $v\ne v_{0}$,
\begin{align*}
(\rho_{v_{0}}\left(f\right))(v) & =\frac{1}{\#\{v':d(v_{0},v)=d(v_{0},v')\}}\sum_{v':d(v_{0},v)=d(v_{0},v')}f(v')\\
 & =\frac{1}{(q+1)q^{d(v_{0},v)-1}}(A_{d(v_{0},v)}f)(v_{0}).
\end{align*}

Since $\tilde{f}_{\theta}$ spans $V_{\theta}$ we have $\left(A_{k}\tilde{f}_{\theta}\right)(v_{0})=A_{k}(\theta)\tilde{f}_{\theta}(v_{0})=A_{k}(\theta)$,
so explicitly $f_{\theta}(v_{0})=1$ and for $v\ne v_{0}$, 
\[
f_{\theta}(v)=\frac{1}{(q+1)q^{d(v_{0},v)-1}}A_{d(v_{0},v)}(\theta).
\]

We claim that $f_{\theta}$ also spans a representation which is isomorphic
to $V_{\theta}$, that is, $A_{k}$ acts on $f_{\theta}$ by $A_{k}(\theta)$.
We call the resulting representation the \emph{spherical model} of
$V_{\theta}$, or the \emph{geometric realization }of $V_{\theta}$.
The claim can be proven directly, but also follows from the following
interesting lemma:
\begin{lem}
\label{lem:Commuting-Lemma}The operator $\rho_{v_{0}}$ commutes
with the action of $H_{0}$ on $\mathbb{C}^{V_{T}}$.
\end{lem}

The intuition for the lemma is that $H_{0}$ commutes with automorphisms
and $\rho_{v_{0}}$ is the ``average'' of all automorphisms fixing
$v_{0}$. Formalizing this intuition is left to the reader.

As for the $L^{p}$-norm of $f_{\theta}$, we have:
\begin{prop}
\label{prop:theta->tempered}Let $\left|\theta\right|\ge q\left|\theta\right|^{-1}$.
Then $f_{\theta}\in L^{p}(V_{T})$ for $2\le p<\infty$ such that
$\left|\theta\right|<q^{\nicefrac{\left(p-1\right)}{p}}$, and $f_{\theta}\not\in L^{p}(V_{T})$
for $2\le p<\infty$ such that $\left|\theta\right|\ge q^{\nicefrac{\left(p-1\right)}{p}}$.
For $p=\infty$, for $\sqrt{q}\le\left|\theta\right|\le q$ it holds
that $f_{\theta}\in L^{\infty}(V_{T})$ and for $\left|\theta\right|>\left|q\right|$
it holds that $f_{\theta}\not\notin L^{\infty}(V_{T})$.
\end{prop}

\begin{proof}
We have $f_{\theta}(v)=\frac{1}{(q+1)q^{d(v_{0},v)-1}}A_{d(v_{0},v)}(\theta)$
and there are $\left(q+1\right)q^{k-1}$ vertices of distance $k$
from $v_{0}$. 

First consider $p<\infty$. Then $\left\Vert f_{\theta}\right\Vert _{p}^{p}=1+\sum_{k\ge1}\left(q+1\right)q^{k-1}\left(\frac{1}{(q+1)q^{k-1}}\left|A_{k}(\theta)\right|\right)^{p}$. 

Write $a_{k}=\left(\left(q+1\right)q^{k-1}\right)^{1-p}\left|A_{k}(\theta)\right|^{p}=\left(\left(q+1\right)q^{-1}q^{k}\right)^{1-p}\left|A_{k}(\theta)\right|^{p}$
for the $k$-th element of the resulting series. Let $C=\left(\left(q+1\right)q^{-1}\right)^{1-p}$.
By Lemma \ref{lem:Growth Lemma} we have for $k\ge1$, $a_{k}\le C\left(k+1\right)^{p}q^{k(1-p)}\left|\theta\right|^{kp}$,
and for an infinite number of $k$-s $a_{k}\ge2^{-3p}Cq^{k(1-p)}\left|\theta\right|^{kp}$.
Therefore if $\left|\theta\right|<q^{\nicefrac{\left(p-1\right)}{p}}$
then $\limsup a_{k}^{1/k}<1$ and $\left\Vert \tilde{f}_{\theta}\right\Vert _{p}^{p}<\infty$,
and if $\left|\theta\right|\ge q^{\nicefrac{\left(p-1\right)}{p}}$
then $\left\Vert f_{\theta}\right\Vert _{p}^{p}=\infty$. 

For $p=\infty$, by Lemma \ref{lem:Growth Lemma}, $\left|f_{\theta}\left(v\right)\right|$
is not bounded for $\left|\theta\right|>q$, so $f_{\theta}\notin L^{\infty}(V_{T})$.
Since for $p<\infty$, $L^{p}\left(V_{T}\right)\subset L^{\infty}(V_{T})$,
it remains to consider $\left|\theta\right|=q$. Then $\left|A_{k}\left(\theta\right)\right|\le A_{k}\left(q\right)=\left(q+1\right)q^{k-1}$,
so $\left|f_{\theta}\left(v\right)\right|\le1$ and $f_{\theta}\in L^{\infty}(V_{T})$. 
\end{proof}
The calculations motivate us to define:
\begin{defn}
Given a representation $\left(\pi,V\right)$ of $H_{0}$, $u\in V$
and $\varphi\in V^{\ast}$ we call the linear functional $c_{\varphi,u}\colon H_{0}\rightarrow\mathbb{C}$,
$c_{\varphi,u}(h)=\left\langle \varphi,\pi\left(h\right)u\right\rangle $
a \emph{matrix coefficient} of $V$. 

For every matrix coefficient we associate a \emph{geometric realization}
$f_{\varphi,u}^{v_{0}}=f_{\varphi,u}\in\mathbb{C}^{V_{T}}$ given
by 
\[
f_{\varphi,u}(v)=\frac{1}{(q+1)q^{d(v_{0},v)-1}}\left\langle \varphi,\pi\left(A_{d(v_{0},v)}\right)u\right\rangle .
\]

We say that $V$ is \emph{$p$-finite }if for every $u\in V$ and
$\varphi\in V^{\ast}$, we have that $f_{\varphi,u}(v)\in L^{p}(V_{T})$,
or equivalently (for $p<\infty$) 
\[
1+\sum_{k\ge1}\left((q+1)q^{k-1}\right)^{1-p}\left|\left\langle \varphi,\pi\left(A_{k}\right)u\right\rangle \right|^{p}<\infty.
\]

We say that $V$ is \emph{$p$-tempered} if it is $p'$-finite for
every $p'>p$. 
\end{defn}

Notice that $\infty$-finite representations are also $\infty$-tempered
by this definition.

The representation $V_{\theta}$ is one dimensional, so it has only
one matrix coefficient up to scale. We can then conclude:
\begin{cor}
\label{cor:theta to tempered}The representation $V_{\theta}$ is
$p$-tempered if and only if $\max\left\{ \left|\theta\right|,q\left|\theta\right|^{-1}\right\} \le q^{\nicefrac{\left(p-1\right)}{p}}$. 

A finite dimensional representation which is a direct sum of irreducible
representations $V=\oplus_{i}V_{\theta_{i}}$ is $p$-tempered if
and only if each $V_{\theta_{i}}$ is $p$-tempered.
\end{cor}

\begin{rem}
Matrix coefficient are very standard in representation theory, and
in particular in representation theory of $p$-adic algebraic groups.
The notion of geometric realization is not standard. The notions of
sectorial model, and spherical model are not standard, and correspond
to the Poincaré disk model and Poincaré half-plane model of the hyperbolic
plane. The notion of a tempered representation is usually refers to
what we call a $2$-tempered representation, and the notion of a $p$-tempered
representation is not standard. 
\end{rem}

\subsection{\label{subsec:Action-on-Finite-graphs}Action on Finite Graphs}

Let $X$ be a finite, connected, $\left(q+1\right)$-regular graph.
As $H_{0}$ is freely generated by $A$, an action of an operator
$A_{0}$ on a vector space $V$ extends to a representation $\left(\pi,V\right)$
of $H_{0}$, given by $\pi(A)=A_{0}$. Therefore the standard action
of the vertex adjacency operator $A_{X}$ on $\mathbb{C}^{V_{X}}\cong L^{2}(V_{X})$
extends to a representation $\left(\pi_{X},L^{2}(V_{X})\right)$ of
$H_{0}$, given by $\pi_{X}\left(A\right)=A_{X}$. Moreover, with
respect to the standard $L^{2}$-norm on $V_{X}$ the operator $A_{X}$
is self-adjoint, so it is diagonalizable and has real eigenvalues.
By looking at the maximal value in absolute value of an eigenvector
of $A_{X}$, each eigenvalue of $A_{X}$ is bounded in absolute value
by $q+1$, so its spectrum is within the range $\left[-q-1,q+1\right]$.
Therefore the representation $\left(\pi_{X},L^{2}(V_{X})\right)$
of $H_{0}$ is a finite direct sum of one dimensional representations,
and each Satake parameter $\theta$ of such a representation satisfies
that $\lambda=\theta+q\theta^{-1}$ is real and of absolute value
$\le q+1$.

Solving $\theta,\tilde{\theta}=\frac{\lambda\pm\sqrt{\lambda^{2}-4q}}{2}$,
we have two options: either $\left|\lambda\right|\le2\sqrt{q}$ (the
Ramanujan range) in which case $\left|\theta\right|=\sqrt{q}$ and
$V_{\theta}$ is 2-tempered, or $2\sqrt{q}\le\left|\lambda\right|\le q+1$,
$\theta$ is real and $1\le\left|\theta\right|\le q$ (i.e.~$-q\le\theta\le-1$
or $1\le\theta\le q$), and $V_{\theta}$ is $p$-tempered for $p$
such that $\max\{\left|\theta\right|,q\left|\theta\right|^{-1}\}\le q^{\nicefrac{\left(p-1\right)}{p}}$.

The eigenvalue $q+1$ for $A$ is only achieved on constant functions
on $X$. Similarly, the eigenvalue $-q-1$ appears if and only if
$X$ is bipartite and is achieved on functions that are constant on
each side of the graph, with one side negative of the other. Ignoring
these two cases we seek the decomposition of the $H_{0}$-representation
$\left(\pi_{X},L_{00}^{2}(V_{X})\right)$, as in the introduction.
We can now prove Theorem \ref{thm:Basic graph claim}:
\begin{proof}[Proof of Theorem \ref{thm:Basic graph claim}]
Let $\lambda$ be the largest absolute value of an eigenvalue of
$A_{X}$ on $L_{00}^{2}(V_{X})$. By Corollary \ref{cor:theta to tempered},
$\lambda\le q^{\nicefrac{1}{p}}+q^{\nicefrac{\left(p-1\right)}{p}}$
if and only if $L_{00}^{2}(V_{X})$ is $p$-tempered. 

Note that the function $\rho_{v_{0}}(\tilde{f})$ as in the introduction
is a special case of a geometric realization of $L_{00}^{2}(V_{X})$.
Let $\tilde{v}_{0}$ be a projection of $v_{0}\in V_{T}$ to $V_{X}$,
and denote by $\boldsymbol{1}_{\tilde{v}_{0}}\in L^{2}(V_{X})$ the
function whose value is $1$ on $\tilde{\ensuremath{v}}_{0}$ and
$0$ elsewhere. We have that for $f\in L_{00}^{2}(V_{X})$
\[
c_{\boldsymbol{1}_{\tilde{v}_{0}},f}(h)=\left\langle \boldsymbol{1}_{\tilde{v}_{0}},\pi_{X}\left(h\right)f\right\rangle =\left(\pi_{X}\left(h\right)f\right)\left(\tilde{v}_{0}\right)
\]
is a matrix coefficient, and its corresponding geometric realization
$f_{\boldsymbol{1}_{\tilde{v}_{0}},f}^{v_{0}}$ equals $\rho_{v_{0}}\left(\tilde{f}\right)$
(note that $\boldsymbol{1}_{\tilde{v}_{0}}\notin L_{00}^{2}\left(V_{X}\right)$,
but still defines a linear functional on $L_{00}^{2}\left(V_{X}\right)$).
Therefore if $\lambda\le q^{\nicefrac{1}{p}}+q^{\nicefrac{\left(p-1\right)}{p}}$
then $\rho_{v_{0}}\left(\tilde{f}\right)\in L^{p+\epsilon}(V_{T})$
for every $\epsilon>0$.

As for the other implication, notice that every matrix coefficient
of $L_{00}^{2}(V_{X})$ is a finite linear sum of matrix coefficients
of the form $c_{\boldsymbol{1}_{\tilde{v}_{0}},\tilde{f}}(h)=\left(\pi_{X}\left(h\right)\tilde{f}\right)(\tilde{v}_{0})$,
for $f\in L_{00}^{2}(V_{X})$, $\tilde{v}_{0}\in V_{X}$. Therefore
if $\rho_{v_{0}}\left(\tilde{f}\right)\in L^{p+\epsilon}(V_{T})$
for every $v_{0}\in V_{T}$, $f\in L_{00}^{2}(V_{X})$ then every
geometric realization of $L_{00}^{2}(V_{X})$ is in $L^{p+\epsilon}(V_{T})$
and $L_{00}^{2}(V_{X})$ is $p$-tempered.
\end{proof}

\subsection{\label{subsec:L_p-Spectrum-of-Hecke}The $L^{p}$-Spectrum of Hecke
Operators}

In this subsection we explain the connection between the notion of
$p$-temperedness and the spectrum of Hecke operators on $L^{p}(V_{T})$. 

Recall that the \emph{norm} bounded operator $h$ on a Banach space
$V$ is $\left\Vert h\right\Vert =\sup_{\left\Vert v\right\Vert =1}\left\Vert hv\right\Vert $.
The \emph{eigenvalues} of $h$ is the set of $\lambda\in\mathbb{C}$
such that there exists $0\ne v\in V$, with $hv=\lambda v$. The \emph{approximate
point spectrum} of $h$ is the set of $\lambda\in\mathbb{C}$ such
that for every $\epsilon>0$ there exists $0\ne v\in V$, with $\left\Vert h-\lambda v\right\Vert <\epsilon\left\Vert v\right\Vert $.
The \emph{spectrum} of $h$ is the set of $\lambda\in\mathbb{C}$
such that $h-\lambda$ has no bounded inverse. The \emph{residual
spectrum} is the complement in the spectrum of the approximate point
spectrum. It is well known that the norm of $h$ bounds the absolute
value of every $\lambda$ in its spectrum.

We will need the following lemmas in our calculations:
\begin{lem}
\label{lem:sum of numbers lemma}Let $x_{1},...,x_{m}\in\mathbb{C}$.
Then $\left|\sum_{i=1}^{m}x_{i}\right|^{p}\le m^{p-1}\sum_{i=1}^{m}\left|x_{i}\right|^{p}$
for every $p\ge1$, with an equality if all the numbers are equal.
\end{lem}

\begin{proof}
First, $\left|\sum_{i=1}^{m}x_{i}\right|^{p}\le\left(\sum_{i=1}^{m}\left|x_{i}\right|\right)^{p}$.
By the convexity of $f(x)=x^{p}$ in $\mathbb{R}_{\ge0}$ we have
$\left(\frac{1}{m}\sum_{i=1}^{m}\left|x_{i}\right|\right)^{p}\le\frac{1}{m}\sum_{i=1}^{m}\left|x_{i}\right|^{p}$.
The equality part is trivial.
\end{proof}
\begin{lem}
\label{lem:sum of numbers lemma-1}Let $X=X_{0}\cup X_{1}$ be a biregular
graph, such that every $x\in X_{0}$ is connected to $K_{0}$ vertices
in $X_{1}$, and every $y\in X_{1}$ is connected to $K_{1}$ vertices
in $X_{0}$. 

Let $\tilde{A}\colon\mathbb{C}^{X_{0}}\rightarrow\mathbb{C}^{X_{1}}$
be the adjacency operator from $X_{0}$ to $X_{1}$, i.e.~$\tilde{A}f(y)=\sum_{x\sim y}f(x)$.
Then as an operator $\tilde{A}\colon L^{p}(X_{0})\rightarrow L^{p}(X_{1})$,
we have $\left\Vert \tilde{A}\right\Vert _{p}\le K_{0}^{\nicefrac{1}{p}}K_{1}^{\nicefrac{\left(p-1\right)}{p}}$,
with an equality if the graph is finite.
\end{lem}

\begin{proof}
For $f\in L^{p}(X_{0})$, we have 
\begin{align*}
\left\Vert \tilde{A}f\right\Vert _{p}^{p} & =\sum_{y\in X_{1}}\left|\tilde{A}f(y)\right|^{p}=\sum_{y\in X_{1}}\left|\sum_{x\sim y}f(x)\right|^{p}\le\sum_{y\in X_{1}}K_{1}^{p-1}\sum_{x\sim y}\left|f(x)\right|^{p}\\
 & =K_{1}^{p-1}\sum_{x\in X_{0}}\left|f(x)\right|^{p}\sum_{y\sim x}1=K_{1}^{p-1}K_{0}\left\Vert f\right\Vert _{p}^{p}.
\end{align*}

The inequality is a result of Lemma \ref{lem:sum of numbers lemma}.
It is an equality if $f$ is constant, and if the graph is finite
such a function is in $L^{p}(X_{0})$.
\end{proof}
The following proposition shows that the $L^{p}$-spectrum of Hecke
operators must contain certain elements.
\begin{prop}
\label{prop:Temperd derives wea containment}Let $h\in H_{0}$. If
$V_{\theta}$ is $p$-tempered then $h(\theta)$ is an eigenvalue
of $h$ on $L^{p'}(V_{T})$ for every $p'>p$, and $h(\theta)$ is
in the approximate point spectrum of $h$ on $L^{p}(V_{T})$.
\end{prop}

\begin{proof}
The geometric realization of $V_{\theta}$ provides us with a function
$f_{\theta}\in\cap_{p'>p}L^{p'}(V_{T})$ which is an eigenvector of
$h$ with eigenvalue $h(\theta)$. The first claim follows.

For the second claim, let $\epsilon>0$ and define $f_{\theta}^{\epsilon}\in\mathbb{C}^{V_{T}}$
by $f_{\theta}^{\epsilon}(v)=f_{\theta}(v)(1-\epsilon)^{d(v,v_{0})}$.
We claim that $f_{\theta}^{\epsilon}\in L^{p}(V_{T})$. Using the
same arguments as in Proposition \ref{prop:theta->tempered}, let
$a_{k}$, $a_{k}^{\epsilon}$ be the $k$-th elements in the series
in the calculations of $\left\Vert f_{\theta}\right\Vert _{p}^{p}$,
$\left\Vert f_{\theta}^{\epsilon}\right\Vert _{p}^{p}$. Then $\limsup a_{k}^{1/k}\le1$.
Since $a_{k}^{\epsilon}=a_{k}(1-\epsilon)^{kp}$, $\limsup\left(a_{k}^{\epsilon}\right)^{1/k}=(1-\epsilon)^{p}\limsup a_{k}^{1/k}<1$,
and the series of $\left\Vert f_{\theta}^{\epsilon}\right\Vert _{p}^{p}$
converges.

Let us calculate $\left\Vert Af_{\theta}^{\epsilon}-A(\theta)f_{\theta}^{\epsilon}\right\Vert _{p}$.
Assume that $\epsilon<\nicefrac{1}{2}$. For $v\in V_{T}$,
\begin{eqnarray*}
\left|\left(Af_{\theta}^{\epsilon}-A(\theta)f_{\theta}^{\epsilon}\right)(v)\right|^{p} & = & \left|\left(Af_{\theta}^{\epsilon}-(1-\epsilon)^{d(v,v_{0})}Af_{\theta}\right)(v)+\left((1-\epsilon)^{d(v,v_{0})}Af_{\theta}-A(\theta)f_{\theta}^{\epsilon}\right)(v)\right|^{p}\\
 & = & \left|\left(Af_{\theta}^{\epsilon}-(1-\epsilon)^{d(v,v_{0})}Af_{\theta}\right)(v)+0\right|^{p}\\
 & = & \left|\sum_{v'\sim v}\left(f_{\theta}^{\epsilon}(v')-(1-\epsilon)^{d(v,v_{0})}f_{\theta}(v')\right)\right|^{p}\\
 & = & \left|\sum_{v'\sim v}\left((1-\epsilon)^{d(v',v_{0})}f_{\theta}(v')-(1-\epsilon)^{d(v,v_{0})}f_{\theta}(v')\right)\right|^{p}\\
 & \le & (q+1)^{p-1}\sum_{v'\sim v}\left|1-(1-\epsilon)^{d(v,v_{0})-d(v',v_{0})}\right|^{p}\left|f_{\theta}^{\epsilon}(v')\right|^{p}\\
 & \le & (q+1)^{p-1}2^{p}\epsilon^{p}\sum_{v'\sim v}\left|f_{\theta}^{\epsilon}(v')\right|^{p}.
\end{eqnarray*}

The first inequality follows from Lemma \ref{lem:sum of numbers lemma}.
For the second inequality, since $\left|d(v,v_{0})-d(v',v_{0})\right|=1$
and $\epsilon<\nicefrac{1}{2}$, we have $\left|1-(1-\epsilon)^{d(v,v_{0})-d(v',v_{0})}\right|\le2\epsilon$.

Summing over all $v\in V_{T}$, we get $\left\Vert Af_{\theta}^{\epsilon}-A(\theta)f_{\theta}^{\epsilon}\right\Vert _{p}^{p}<C\epsilon^{p}\left\Vert f_{\theta}^{\epsilon}\right\Vert _{p}^{p}$,
for $C=2^{p}\left(q+1\right)^{p}$. Therefore as $\epsilon\rightarrow0$,
the $f_{\theta}^{\epsilon}$ are approximate eigenvectors for the
approximate eigenvalue $A(\theta)$ of $A$. 

Finally, since $A$ generates $H_{0}$, $h(\theta)$ is an approximate
eigenvalue of $h$, with $f_{\theta}^{\epsilon}$ as approximate eigenvectors.
\end{proof}
The following proposition bounds the $L^{p}$-norm of Hecke operators:
\begin{prop}
\label{prop:Kesten's}Let $p\ge2$. The norm (and therefore the absolute
value of every element of the spectrum) of $A_{k}$ on $L^{p}(V_{T})$
is bounded by $\left\Vert A_{k}\right\Vert _{p}\le A_{k}\left(q^{\nicefrac{\left(p-1\right)}{p}}\right)\le\left(k+1\right)q^{\nicefrac{k\left(p-1\right)}{p}}$.

In particular, the norm of $A$ is bounded by $\left\Vert A\right\Vert _{p}\le A\left(q^{\nicefrac{\left(p-1\right)}{p}}\right)=q^{\nicefrac{1}{p}}+q^{\nicefrac{\left(p-1\right)}{p}}$.
\end{prop}

\begin{proof}
Consider an infinite ray $R$ on the tree. Recall from the discussion
in Subsection \ref{subsec:Geometric-Realization} that every vertex
$v$ has one neighbor $u_{0}^{v}$ with relative distance $c(u_{0}^{v})=c(v)-1$
and $q$ neighbors $u_{1}^{v},...,u_{q}^{v}$ with $c(u_{i}^{v})=c(v)+1$.

Define $h_{0},h_{1}\colon\mathbb{C}^{V_{T}}\rightarrow\mathbb{C}^{V_{T}}$
as follows: let $f\in\mathbb{C}^{V_{T}}$. Then $h_{0}f(v)=\sum_{i=1}^{q}f(u_{i}^{v})$,
i.e.~the sum of $f$ on the $q$ vertices that have greater relative
distance. Similarly, $h_{1}f(v)=f(u_{0}^{v})$ is the value of $f$
on the single neighbor of $v$ that has shorter relative distance.

We first prove the proposition for $A$. By definition, we have $A=h_{0}+h_{1}$.
We claim that 
\begin{eqnarray*}
\left\Vert h_{0}f\right\Vert _{p} & \le & q^{\nicefrac{\left(p-1\right)}{p}}\left\Vert f\right\Vert _{p}\\
\left\Vert h_{1}f\right\Vert _{p} & = & q^{\nicefrac{1}{p}}\left\Vert f\right\Vert _{p}.
\end{eqnarray*}

The equality is immediate, since every value in (the series of) $\left\Vert h_{1}f\right\Vert _{p}^{p}$
is a value in (the series of) $\left\Vert f\right\Vert _{p}^{p}$,
while each value in $\left\Vert f\right\Vert _{p}^{p}$ appears $q$
times in $\left\Vert h_{1}f\right\Vert _{p}^{p}$. The inequality
follows from Lemma \ref{lem:sum of numbers lemma}, since each value
in $\left\Vert h_{0}f\right\Vert _{p}^{p}$ is a sum of $q$ values
in $\left\Vert f\right\Vert _{p}^{p}$, and each value in $\left\Vert f\right\Vert _{p}^{p}$
appears in exactly one such sum. 

Therefore $\left\Vert A\right\Vert _{p}\le\left\Vert h_{0}\right\Vert _{p}+\left\Vert h_{1}\right\Vert _{p}\le q^{\nicefrac{1}{p}}+q^{\nicefrac{\left(p-1\right)}{p}}$.

The proof for $A_{k}$ is a direct generalization: we can write $A_{k}=h_{0}+....+h_{k}$,
where:
\begin{itemize}
\item $h_{0}f(v)$ is the sum of $f$ on the $q^{k}$ vertices $u$, with
$d(v,u)=k$ and $c(u)-c(v)=k$ 
\item $h_{k}f(v)$ is the value of $f$ on the single vertex $u$, with
$d(v,u)=k$ and $c(u)-c(v)=-k$.
\item $h_{i}f(v)$, $0<i<k$ is the sum of $f$ on the $\left(q-1\right)q^{k-i-1}=\left(1-q^{-1}\right)q^{k-i}$
vertices $u$ with $d(v,u)=k$ and $c(u)-c(v)=k-2i$.
\end{itemize}
Write for simplicity $\theta_{p}=q^{\nicefrac{\left(p-1\right)}{p}}$.
Then:
\begin{eqnarray*}
\left\Vert h_{0}\right\Vert _{p} & \le & q^{\nicefrac{k\left(p-1\right)}{p}}=\theta_{p}^{k}\\
\left\Vert h_{k}\right\Vert _{p} & = & q^{\nicefrac{k}{p}}=\left(q\theta_{p}^{-1}\right)^{k}\\
\left\Vert h_{i}\right\Vert _{p} & \le & (1-q^{-1})q^{\nicefrac{\left(k-i\right)(p-1)}{p}}q^{\nicefrac{i}{p}}=(1-q^{-1})\theta_{p}^{k-i}\left(q\theta_{p}^{-1}\right)^{i}.
\end{eqnarray*}

The bounds for $h_{0}$ and $h_{k}$ are proved similarly to the bounds
in the calculations for $A$. Let us prove the bounds for $0<i<k$:
build a bipartite (infinite) directed graph $G_{i}$, with $X_{0}=V_{T}\times\{0\}$,
$X_{1}=V_{T}\times\{1\}$. Connect $\left(u,0\right)$ to $\left(v,1\right)$
if $d(v,u)=k$ and $c(u)-c(v)=k-2i$. Then the adjacency operator
from $\mathbb{C}^{X_{0}}$ to $\mathbb{C}^{X_{1}}$ acts exactly like
the operator $h_{i}$ acts on $\mathbb{C}^{V_{T}}$. With the notations
of Lemma \ref{lem:sum of numbers lemma-1}, have that $K_{0}=(q-1)q^{i-1}=\left(1-q^{-1}\right)q^{i}$
and $K_{1}=(q-1)q^{k-i-1}=\left(1-q^{-1}\right)q^{k-i}$. Now apply
Lemma \ref{lem:sum of numbers lemma-1} and organize to arrive to
the given bounds.

Therefore,
\begin{align*}
\left\Vert A_{k}\right\Vert _{p} & \le\left\Vert h_{0}\right\Vert _{p}++...+\left\Vert h_{k}\right\Vert _{p}\\
 & \le\theta_{p}^{k}+\left(q\theta_{p}^{-1}\right)^{k}+\sum_{i=1}^{k-1}\left(1-q^{-1}\right)\theta_{p}^{k-i}\left(q\theta_{p}^{-1}\right)^{i}\\
 & =A_{k}\left(\theta_{p}\right)=A_{k}\left(q^{\nicefrac{\left(p-1\right)}{p}}\right).
\end{align*}

We finish by using Lemma \ref{lem:Growth Lemma}.
\end{proof}
We can now prove Theorem \ref{thm:Lp bound}:
\begin{cor}
For $2\le p\le\infty$, the spectrum of $A$ on $L^{p}(V_{T})$ is
$\left\{ \theta+q\theta^{-1}:\theta\in\mathbb{C},\,q^{\nicefrac{1}{p}}\le\left|\theta\right|\le q^{\nicefrac{\left(p-1\right)}{p}}\right\} $,
and each point of it belongs to the approximate point spectrum. For
$p=\infty$ each element of the spectrum is an eigenvalue. For $2\le p<\infty$
the set of eigenvalues is the interior $\left\{ \theta+q\theta^{-1}:\theta\in\mathbb{C},\,q^{\nicefrac{1}{p}}<\left|\theta\right|<q^{\nicefrac{\left(p-1\right)}{p}}\right\} $.

For $1\le p<2$ the spectrum is the same as for $p/\left(p-1\right)$.
For $1<p\le2$ the interior $\left\{ \theta+q\theta^{-1}:\theta\in\mathbb{C},\,q^{\nicefrac{\left(p-1\right)}{p}}<\left|\theta\right|<q^{\nicefrac{1}{p}}\right\} $
is in the residual spectrum, and the boundary $\left\{ \theta+q\theta^{-1}:\theta\in\mathbb{C},\,\left|\theta\right|=q^{\nicefrac{\left(p-1\right)}{p}}\right\} $
belongs to the approximate point spectrum but is not an eigenvalue.
For $p=1$ the entire spectrum belongs to the residual spectrum.
\end{cor}

\begin{proof}
Assume $p\ge2$. By Corollary \ref{cor:theta to tempered} and Proposition
\ref{prop:Temperd derives wea containment} every point in the interior
is an eigenvalue and every point in the boundary is in the approximate
point spectrum.

For $2\le p<\infty$ points in the boundary are not eigenvalues since
if $f\in L^{p}(V_{T})$ is an eigenvector with eigenvalue $A(\theta$),
and $f(v_{0})\ne0$, then also $\rho_{v_{0}}f\in L^{p}(V_{T})$ and
$\rho_{v_{0}}f(v_{0})\ne0$. But then necessarily $\rho_{v_{0}}f=f(v_{0})f_{\theta}$
and $f_{\theta}\notin L^{p}(V_{T})$ for $\theta$ in the boundary
by Corollary \ref{cor:theta to tempered}. If $A(\theta)\in\mbox{Spec}_{L^{p}(V_{T})}A$,
then by the Satake isomorphism $A_{k}(\theta)\in\mbox{Spec}_{L^{p}(V_{T})}A_{k}$
for every $k\ge1$ and in particular the norm of $A_{k}$ on $L^{p}(V_{T})$
is at least $\left|A_{k}(\theta)\right|$. Assume $\left|\theta\right|\notin\left[q^{\nicefrac{1}{p}},q^{\nicefrac{\left(p-1\right)}{p}}\right]$
and $\left|\theta\right|\ge q\left|\theta\right|^{-1}$, then $\left|\theta\right|>q^{\nicefrac{\left(p-1\right)}{p}}$.
By Lemma \ref{lem:Growth Lemma}b there exist infinitely many $k>0$,
with $\left\Vert A_{k}\right\Vert _{p}\ge\left|A_{k}(\theta)\right|\ge0.001\left|\theta\right|^{k}$.
By Proposition \ref{prop:Kesten's}, $\left\Vert A_{k}\right\Vert _{p}\le\left(k+1\right)\left(q^{\nicefrac{\left(p-1\right)}{p}}\right)^{k}$,
and since $\left|\theta\right|>q^{\nicefrac{\left(p-1\right)}{p}}$
we have a contradiction.

For $1\le p<2$, note that the action of $A$ on $L^{p/\left(p-1\right)}(V_{T})$
is the dual of the action of $A$ on $L^{p}(V_{T})$ (this is also
true for $p=1$, $p/\left(p-1\right)=\infty$). There are no eigenvalues
since $L^{p}(V_{T})\subset L^{2}(V_{T})$ and there are no eigenvectors
for $L^{2}(V_{T})$. By basic facts of spectral theory, the spectrum
of an operator is equal to the spectrum of its dual. Moreover, for
reflexive Banach spaces, the discrete spectrum (i.e.~the eigenvalues
union the residual spectrum), and the continous spectrum (i.e.~the
approximate point spectrum without the eigenvalues) of dual operators
agree. As there are no eigenvalues, for $1<p<2$ the interior is in
the residual spectrum and the boundary is in the continuous spectrum.
For $p=1$ one uses the fact that the residual spectrum of an operator
without eigenvalues is equal to the set of eigenvalues of its dual.
\end{proof}
\begin{rem}
For $p=2$, Proposition \ref{prop:Temperd derives wea containment}
and Proposition \ref{prop:Kesten's} are versions of Theorem 1 and
Theorem 2 of \cite{cowling1988almost}. The proof of Proposition \ref{prop:Kesten's}
is based on the proof of Theorem 2 in \cite{cowling1988almost}. A
similar combinatorial proof for $p=2$ is given in \cite{angel2015non}
Theorem 4.2.
\end{rem}

\section{\label{sec:Operators-on-Edges}Operators on Edges}

\subsection{\label{subsec:The-Directed-Edge-Hecke}The Iwahori-Hecke Algebra}

We wish to extend the $L^{p}$-theory to operators acting on the directed
edges of the tree or the graph. The theory here is slightly more complicated,
since the algebra is not commutative and the operators are not self-adjoint.
Since the proofs are very similar to the vertex case, some of them
are omitted. In any case, a generalized full treatment is given in
\cite{kamber2016lpcomplex}. 

We denote by $E_{T}$ the directed edges of the tree and by $E_{X}$
the directed edges of the (finite, non-oriented) graph $X$ from the
introduction. Each non-oriented edge is counted twice in $E_{T}$
and $E_{X}$.
\begin{defn}
Let $h_{s_{0}},h_{s_{1}},h_{\tau},h_{NB}\colon\mbox{\ensuremath{\mathbb{C}}}^{E_{T}}\rightarrow\mbox{\ensuremath{\mathbb{C}}}^{E_{T}}$
be the following operators:
\begin{eqnarray*}
h_{s_{0}}\tilde{f}(x,y) & = & \sum_{y'\sim x,y'\ne y}\tilde{f}(x,y')\\
h_{s_{1}}\tilde{f}(x,y) & = & \sum_{x'\sim y,x'\ne x}\tilde{f}(x',y)\\
h_{\tau}\tilde{f}(x,y) & = & \tilde{f}(y,x)\\
h_{NB}\tilde{f}(x,y) & = & h_{\tau}h_{s_{0}}\tilde{f}=h_{s_{1}}h_{\tau}\tilde{f}=\sum_{x'\sim y,x'\ne x}\tilde{f}(y,x').
\end{eqnarray*}

The \emph{Iwahori-Hecke algebra} $H_{\phi}$, or the \emph{directed
edge Hecke algebra} is the algebra of operators acting on $\mathbb{C}^{E_{T}}$
generated by the operators $h_{s_{0}},h_{s_{1}}$ and $h_{\tau}$. 
\end{defn}

We will show in the beginning of Subsection \ref{subsec:Representation-Theory-of}
that there is a natural representation $\left(\pi_{X},L^{2}\left(E_{X}\right)\right)$
of $H_{\phi}$. For now we analyse the action of $H_{\phi}$ on $\mbox{\ensuremath{\mathbb{C}}}^{E_{T}}$.
Notice, however, that the operator $\pi_{X}\left(h_{NB}\right)$ is
Hashimoto's non-backtracking operator (see \cite{hashimoto1989zeta}).
The non-backtracking operator is used in the theory of the graph Zeta
function, which is defined as 
\[
\zeta_{X}(u)=\frac{1}{\det\left(1-u\pi_{X}\left(h_{NB}\right)\right)}.
\]

Our discussion here is indeed similar to the discussion of Hashimoto
on the Zeta function in \cite{hashimoto1989zeta}.
\begin{defn}
Let $\left(\hat{W},S\right)$ be the extended Coxeter group $\hat{W}=\left\langle \left.s_{0},s_{1},\tau\right|s_{0}^{2}=s_{1}^{2}=\tau^{2}=1,\,\tau s_{0}=s_{1}\tau\right\rangle $,
with its set of generators $S=\left\{ s_{0},s_{1},\tau\right\} $.

Let $w_{NB}\in\hat{W}$ be the element $\tau s_{0}=s_{1}\tau$.
\end{defn}

\begin{lem}
\label{lem:Coxeter_group_structure}Each $w\in\hat{W}$ can be written
uniquely as $w=\tau^{\delta_{\tau}}w_{NB}^{m}s_{1}^{\delta_{1}}$,
for $\delta_{\tau},\delta_{1}\in\{0,1\}$ and $m\ge0$.
\end{lem}

\begin{proof}
By the relations involving $\tau$ every $w\in\hat{W}$ may be written
uniquely as $w=\tau^{\delta_{\tau}^{\prime}}w'$ where $\delta_{\tau}^{\prime}\in\{0,1\}$
and $w'$ is a product of $s_{0}$ and $s_{1}$ only. Since $s_{0}^{2}=s_{1}^{2}=1$,
$w'$ may be written uniquely as $w'=s_{0}^{\alpha_{0}}\left(s_{1}s_{0}\right)^{m'}s_{1}^{\delta_{1}}$,
with $m'\ge0$, $\alpha_{0},\delta_{1}\in\{0,1\}$. Since $w_{NB}^{2}=s_{1}\tau\tau s_{0}=s_{1}s_{0}$,
\begin{align*}
w & =\tau^{\delta_{\tau}^{\prime}}w'=\tau^{\delta_{\tau}^{\prime}}s_{0}^{\alpha_{0}}\left(s_{1}s_{0}\right)^{m^{\prime}}s_{1}^{\delta_{1}}\\
 & =\tau^{\delta_{\tau}^{\prime}}\tau^{\alpha_{0}}w_{NB}^{2m'+\alpha_{0}}s_{1}^{\delta_{1}}=\tau^{\delta_{\tau}}w_{NB}^{m}s_{1}^{\delta_{1}},
\end{align*}
with $\delta_{\tau}\equiv\delta_{\tau}^{\prime}+\alpha_{0}\mod2$,
$m=2m'+\alpha_{0}$.

As one may recover $m',\alpha_{0},\delta_{\tau}^{\prime}$ from $m,\delta_{\tau}$,
it also proves uniqueness.
\end{proof}
\begin{defn}
The \emph{Coxeter length function $l\colon\hat{W}\rightarrow\mathbb{N}$
is defined by $l(\tau^{\delta_{\tau}}w_{NB}^{m}s_{1}^{\delta_{1}})=m+\delta_{1}$.}
\end{defn}

For $w=\tau^{\delta_{\tau}}w_{NB}^{m}s_{1}^{\delta_{1}}\in\hat{W}$
we denote $h_{w}=h_{\tau}^{\delta_{\tau}}h_{NB}^{m}h_{s_{1}}^{\delta_{1}}$.

Notice that our two different notations for $h_{s_{0}},h_{\tau},h_{s_{1}}$
agree with each other and that $h_{w_{NB}}=h_{NB}$.

For $e\in E_{T}$ we denote by $\boldsymbol{1}_{e}\in\mbox{\ensuremath{\mathbb{C}}}^{E_{T}}$
the function whose value is $1$ on $e$ and $0$ elsewhere.
\begin{lem}
\label{lem:Tree_is_building}Let $e_{0},e_{1}\in E_{T}$. Then:

1. The function $h_{w}\boldsymbol{1}_{e_{1}}$ is non-zero on $q^{l(w)}$
edges.

2. There exists a unique $w\in\hat{W}$ such that $h_{w}\boldsymbol{1}_{e_{1}}$
is non-zero on $e_{0}$, and then $h_{w^{-1}}1_{e_{0}}$ is non-zero
on $e_{1}$.
\end{lem}

\begin{proof}
Proved easily by the decomposition $w=\tau^{\delta_{\tau}}w_{NB}^{m}s_{1}^{\delta_{1}}$
and induction on $l(w)$.
\end{proof}
\begin{defn}
\label{def;Distance definition}For $e_{0},e_{1}\in E_{T}$, the \emph{distance
}$d(e_{0},e_{1})\in\hat{W}$ is the unique $w\in\hat{W}$ such that
$h_{w}1_{e_{1}}$ is supported on $e_{0}$.
\end{defn}

Notice that by Lemma \ref{lem:Tree_is_building}, if $d(e_{0},e_{1})=w$
then $d\left(e_{1},e_{0}\right)=w^{-1}$.

While this definition is not standard in combinatorics, it is standard
when treating the tree $T$ as a building. The abstract reason for
this definition is the following lemma, which is left for the reader.
\begin{lem}
\label{Lem:Distance vs automorphism}Let $e_{0},e_{1},e_{0}^{\prime},e_{1}^{\prime}\in E_{T}$.
Then $d\left(e_{0},e_{1}\right)=d\left(e_{0}^{\prime},e_{1}^{\prime}\right)$
if and only if there exists a tree automorphism $\gamma\in\mbox{Aut}(T)$
such that $\gamma\left(e_{0}\right)=e_{0}^{\prime}$ and $\gamma\left(e_{1}\right)=e_{1}^{\prime}$.
\end{lem}

As a result of Lemma \ref{lem:Tree_is_building} and the definition
of $d(e_{0},e_{1})$, we may write for $w\in\hat{W}$ and $f\in\mathbb{C}^{E_{T}}$:
\[
h_{w}f(e_{0})=\sum_{e_{1}:d(e_{0},e_{1})=w}f(e_{1}).
\]

We can now describe $H_{\phi}$ as follows:
\begin{lem}
\label{lem:H_phi relations}The algebra $H_{\phi}$ is isomorphic
to the algebra defined abstractly by the generating operators $h_{s_{0}},h_{s_{1}},h_{\tau}$,
and the relations:
\begin{eqnarray*}
h_{s_{0}}^{2} & = & q\cdot Id+(q-1)h_{s_{0}}\\
h_{s_{1}}^{2} & = & q\cdot Id+(q-1)h_{s_{1}}\\
h_{\tau}^{2} & = & Id\\
h_{\tau}h_{s_{0}} & = & h_{s_{1}}h_{\tau}.
\end{eqnarray*}

The algebra $H_{\phi}$ is also isomorphic to the algebra which is
the linear span of the basis operators $h_{w}$, $w\in\hat{W}$, with
the relations above and the relation for $w\in\hat{W}$ and $s\in S$
\begin{align*}
h_{w}h_{s}=h_{ws} & \,\,\,\,\,\text{if}\,\,s=\tau\text{ or }l(ws)=l(w)+1.
\end{align*}
\end{lem}

\begin{rem}
\label{rem:Iwahori-Hecke relations}The relations of Lemma \ref{lem:H_phi relations}
are called the \emph{Iwahori-Hecke relations}. It easily follows from
them that for $w,w'\in\hat{W}$ and $s\in S$:
\begin{align*}
h_{ww} & =h_{w}h_{w'}\,\,\,\,\,\text{if}\,\,l(ww')=l(w)+l(w')\\
h_{w}h_{s} & =qh_{w}+\left(q-1\right)h_{ws}\,\,\,\,\,\text{if}\,\,l(ws)=l(w)-1.
\end{align*}

Some of the relations are actually redundant, as $H_{\phi}$ is generated
by $h_{s_{0}},h_{\tau}$ and the two relations $h_{s_{0}}^{2}=q\cdot Id+(q-1)h_{s_{0}}$
and $h_{\tau}^{2}=Id$.
\end{rem}

\begin{proof}
The fact that the two abstract descriptions are isomorphic is standard
and left to the reader (see \cite{macdonald1996affine} for the general
case). So it is enough to prove that $H_{\phi}$ is isomorphic to
the second description, using a linear basis of the algebra. By Lemma
\ref{lem:Tree_is_building} the $h_{w}\in H_{\phi}$, $w\in\hat{W}$
are indeed linearly independent. The first 4 relations may be verified
directly. The last relation follows from Lemma \ref{lem:Tree_is_building}.
\end{proof}
\begin{lem}
\label{lem:involution}The algebra $H_{\phi}$ has an involution (or
adjunction) $\ast\colon H_{\phi}\rightarrow H_{\phi}$ sending each
$\alpha\cdot h_{w}$, $\alpha\in\mathbb{C},\,w\in\hat{W}$ to $\bar{\alpha}\cdot h_{w^{-1}}$.
\end{lem}

\begin{proof}
Consider the natural inner product on $L^{2}\left(E_{T}\right)$.
By Lemma \ref{lem:Tree_is_building}, for each $e_{0},e_{1}\in E_{T}$
and $w\in\hat{W}$ we have $\left\langle h_{w}\boldsymbol{1}_{e_{0}},\boldsymbol{1}_{e_{1}}\right\rangle =\left\langle \boldsymbol{1}_{e_{0}},h_{w^{-1}}\boldsymbol{1}_{e_{1}}\right\rangle $.
It follow that $h_{w^{-1}}$ is the adjoint of $h_{w}$ relatively
to the natural inner product on $L^{2}\left(E_{T}\right)$. As taking
an adjoint is an involution, the claim follows.
\end{proof}
The following proposition, analogous to Proposition \ref{prop:Hecke_commutes},
gives an abstract definition of $H_{\phi}$.
\begin{prop}
\label{prop:Hecke_commutes-1}Let $\gamma\in\mbox{Aut}(T)$ be an
automorphism of the tree. Then $\gamma$ acts naturally on $\mathbb{C}^{E_{T}}$
by $\gamma\cdot f(x,y)=f(\gamma^{-1}x,\gamma^{-1}y)$.

The algebra $H_{\phi}$ is the algebra of row and column finite operators
acting on $\mathbb{C}^{E_{T}}$ and commuting with tree automorphisms. 
\end{prop}

\begin{proof}
As in Proposition \ref{prop:Hecke_commutes}, making use of Lemma
\ref{Lem:Distance vs automorphism}.
\end{proof}
It is natural to study $H_{0}$ and $H_{\phi}$ together. We can do
it by defining a larger algebra containing them both. Following Proposition
\ref{prop:Hecke_commutes} and Proposition \ref{prop:Hecke_commutes-1},
one can define:
\begin{defn}
\label{def:Full_Hecke_Algebra}The \emph{full graph Hecke algebra
$H$} is the algebra of row and column finite operators acting on
$\mathbb{C}^{E_{T}}\oplus\mathbb{C}^{V_{T}}$ and commuting with tree
automorphisms. 
\end{defn}

Consider the composition: $\mathbb{C}^{E_{T}}\oplus\mathbb{C}^{V_{T}}\overset{p}{\to}\mathbb{C}^{V_{T}}\overset{A_{m}}{\to}\mathbb{C}^{V_{T}}\overset{i}{\to}\mathbb{C}^{E_{T}}\oplus\mathbb{C}^{V_{T}}$,
where $p$ is the projection and $i$ is the extension by zeros. Using
it, $A_{m}$ extends to an operator acting on $\mathbb{C}^{E_{T}}\oplus\mathbb{C}^{V_{T}}$,
which we will denote by abuse of notations $A_{m}$ again. This extension
belongs to the full graph Hecke algebra $H$. Similarly, and again
by abuse of notations, we may extend each operator $h\in H_{\phi}$
to an operator $h\in H$. In other words, the algebras $H_{0},H_{\phi}$
occur as subalgebras (with different units) of $H$ and $Id_{H}=Id_{H_{0}}+Id_{H_{1}}$.

Define the operators: $u\colon\mathbb{C}^{V_{T}}\to\mathbb{C}^{E_{T}}$,
$d\colon\mathbb{C}^{E_{T}}\to\mathbb{C}^{V_{T}}$ by 
\begin{align*}
uf(x,y) & =f(x)\\
df(x) & =\sum_{y\sim x}f(x,y),
\end{align*}

and extend them similarly to operators acting on $\mathbb{C}^{E_{T}}\oplus\mathbb{C}^{V_{T}}$.
We have that $u,d\in H$, and the following relations hold:
\begin{align*}
ud & =h_{s_{0}}+Id_{H_{\phi}}\\
du & =(q+1)\text{\ensuremath{Id}}_{H_{0}}\\
A & =dh_{\tau}u,\,\,\,\,\,\,A_{m}=dh_{\tau}h_{NB}^{m-1}u.
\end{align*}

One can give a complete description of $H$, either in terms of generators
of an algebra or in terms of a linear basis. We will only give the
description in terms of generators and relations.
\begin{prop}
\label{prop:H-speciication}The algebra $H$ is isomorphic to the
algebra defined abstractly by the generators $Id_{H_{\phi}}$, $h_{s_{0}}$,
$h_{s_{1}}$, $h_{\tau}$, $d$ and $u$, the generating relations
of $H_{\phi}$ (with $Id_{H_{\phi}}$ instead of $Id$, including
the relations saying that $Id_{H_{\phi}}$ is the identity of $H_{\phi}$),
and the relations 
\begin{align*}
du & =h_{s_{0}}+id_{H_{\phi}}\\
u^{2} & =d^{2}=uh=0
\end{align*}
for any $h\in H_{\phi}$.
\end{prop}

\begin{proof}
Left to the reader.
\end{proof}
\begin{rem}
\label{rem:history remark}This subsection essentially proves that
the $\left(q+1\right)$-regular tree is a building. See \cite{ronan2009lectures}
for an introduction to buildings.

The group $W=\left\langle \left.s_{0},s_{1}\right|s_{0}^{2}=s_{1}^{2}=1\right\rangle $
is the infinite dihedral group, which is an affine Coxeter group,
or an affine Weyl group. Specifying the generators $S=\left\{ s_{0},s_{1}\right\} $
makes $\left(W,S\right)$ into a Coxeter system. Adding an automorphism
$\tau$ which acts on $W$ by $\tau s_{0}=s_{1}\tau$ makes $\hat{W}$,
with the extra data of $\left\{ s_{1},s_{2}\right\} $ and $\tau$
an extended affine Coxeter group, which is more commonly called an
extended affine Weyl group. As is turns out, $\hat{W}$ is generated
by $s_{0},\tau$ and is isomorphic as a group to $W$ with generators
$s_{0},s_{1}$, but they are not isomorphic as Coxeter systems.

Similarly, the Iwahori-Hecke algebra we defined, which is common in
representation theory, is an extended version of the standard Iwahori-Hecke
algebra, which does not include $h_{\tau}$. While the name an ``extended
Iwahori-Hecke algebra'' is appropriate for it, it is called in the
literature either an Iwahori-Hecke algebra, a Hecke algebra, or an
affine Hecke algebra. Standard references for Iwahori-Hecke algebras
and Coxeter groups are \cite{macdonald1996affine,lusztig2003hecke}.

In his fundamental paper \cite{hashimoto1989zeta}, Hashimoto studied
the graph Zeta function using the Iwahori-Hecke algebra and its representation
theory. However, he used its non-extended version, which does not
include $h_{NB}$.

Readers familiar with the Bernstein-Lusztig presentation of the Iwahori-Hecke
algebra may note that $h_{NB}$ plays a crucial role in this presentation,
as it corresponds (in the notations of \cite{macdonald1996affine})
to the operator $Y^{\lambda}$, $\lambda$ a fundamental coweight. 

The full graph Hecke algebra $H$ is not standard in representation
theory.
\end{rem}

\subsection{\label{subsec:Representation-Theory-of}The Representation Theory
of the Iwahori-Hecke Algebra}

Recall from Lemma \ref{lem:involution} that $H_{\phi}$ has an involution
$\ast\colon H_{\phi}\rightarrow H_{\phi}$. Also recall from the Preliminaries
that a representation $\left(\pi,V\right)$ of $H_{\phi}$ is called
\emph{unitary} if there exists an inner product on $V$ satisfying
$\left\langle \pi\left(h\right)v_{1},v_{2}\right\rangle =\left\langle v_{1},\pi\left(h^{*}\right)v_{2}\right\rangle $
for every $v_{1},v_{2}\in V$ and $h\in H$. 

Let $X$ be a $\left(q+1\right)$-regular graph. As $T$ is the universal
cover of $X$ (see \cite[Chapter 6]{hoory2006expander}), we may consider
$X$ as a quotient of $T$ by a discrete cocompact torsion free group
$\Gamma\subset\mbox{Aut}(T)$. Since the action of $H_{\phi}$ on
$\mathbb{C}^{E_{T}}$ commutes with automorphisms we have an action
of $H_{\phi}$ on the $\Gamma$-invariant vectors of $\mathbb{C}^{E_{T}}$,
which we identify with $\mathbb{C}^{E_{X}}\cong L^{2}(E_{X})$, i.e.~functions
on the directed edges of the finite graph. Moreover, this representation
$\left(\pi_{X},L^{2}(E_{X})\right)$ is unitary with respect to the
usual inner product on $L^{2}(E_{X})$.

\subsubsection{Classification of Irreducible Representations}

The following theorem is the Iwahori-Hecke analog of the Satake isomorphism.
\begin{thm}
\label{thm:Classification-Iwahori-Hecke}There exists an embedding
$\Phi\colon H_{\phi}\to M_{2\times2}\left(\mathbb{C}\left(\left[\theta,\theta^{-1}\right]\right)\right)$
($\theta$ indeterminate), given by ($\tilde{\theta}=q\theta^{-1})$:%
\begin{align*}
\Phi\left(h_{\tau}\right) & =\left(\begin{array}{cc}
0 & 1\\
1 & 0
\end{array}\right),\,\Phi\left(h_{s_{0}}\right)=\left(\begin{array}{cc}
0 & \tilde{\theta}\\
\theta & q-1
\end{array}\right),\,\Phi\left(h_{s_{1}}\right)=\left(\begin{array}{cc}
q-1 & \theta\\
\tilde{\theta} & 0
\end{array}\right)\\
\Phi\left(h_{NB}^{k}\right) & =\left(\begin{array}{cc}
\theta^{k} & \left(q-1\right)\left(\theta^{k-1}+\theta^{k-2}\tilde{\theta}+...+\tilde{\theta}^{k-1}\right)\\
0 & \tilde{\theta}^{k}
\end{array}\right),
\end{align*}
and for $w=\tau^{\delta_{\tau}}w_{NB}^{m}s_{1}^{\delta_{1}}$,
\[
\Phi(h_{w})=\Phi\left(h_{\tau}\right)^{\delta_{\tau}}\Phi\left(h_{NB}^{m}\right)\Phi\left(h_{s_{1}}\right)^{\delta_{1}}.
\]

For any specific $0\ne\theta\in\mathbb{C}$, the evaluation map $\tilde{\pi}_{\theta}\colon M_{2\times2}\left(\mathbb{C}\left(\left[\theta,\theta^{-1}\right]\right)\right)\to M_{2\times2}\left(\mathbb{C}\right)$
defines a 2-dimensional representation $\left(\pi_{\theta}^{\phi},U_{\theta}\right)$
of $H_{\phi}$, with $\pi_{\theta}^{\phi}=\tilde{\pi}_{\theta}\circ\Phi$
and $U_{\theta}=\mathbb{C}^{2}$. Every irreducible finite dimensional
representation is isomorphic to a quotient of such a representation.
\end{thm}

\begin{proof}
By Lemma \ref{lem:H_phi relations}, one should prove that the set
$\left\{ \Phi(h_{w}):w\in\hat{W}\right\} $, is linearly independent
(over $\mathbb{C}$) and that the Iwahori-Hecke relations hold. We
leave the verification to the reader.

To prove the statement about every irreducible representation, let
$\left(\pi,V\right)$ be a finite dimensional representation of $H_{\phi}$,
and let $\theta$ be an eigenvalue of $\pi\left(h_{NB}\right)$ with
eigenvector $0\ne v_{0}\in V$. Let $u_{0},u_{1}=\pi_{\theta}\left(\tau\right)u_{0}$
be the standard basis of $U_{\theta}$. Define a linear transformation
$\varphi\colon U_{\theta}\rightarrow V$ by $\varphi\left(u_{0}\right)=v_{0}$
and $\varphi\left(u_{1}\right)=v_{1}=\pi\left(\tau\right)v_{0}$.
The fact that $v_{0}$ is an eigenvector of $\pi\left(h_{NB}\right)$
with eigenvalue $\theta$ and the Iwahori-Hecke relations says that
the following relations hold:
\begin{align*}
\pi\left(h_{\tau}\right)v_{0} & =v_{1},\,\pi\left(h_{\tau}\right)v_{1}=v_{0}\\
\pi\left(h_{s_{1}}\right)v_{1} & =\pi\left(h_{NB}\right)v_{0}=\theta v_{0}\\
\pi\left(h_{s_{1}}\right)v_{0} & =\pi\left(h_{s_{1}}^{2}\right)\theta^{-1}v_{1}=\left(q\cdot Id+\left(q-1\right)\pi\left(h_{s_{1}}\right)\right)\theta^{-1}v_{1}=\left(q-1\right)v_{0}+\tilde{\theta}v_{1}\\
\pi\left(h_{s_{0}}\right)v_{1} & =\pi\left(h_{s_{0}}h_{\tau}\right)v_{0}=\pi\left(h_{\tau}h_{s_{1}}\right)v_{0}=\left(q-1\right)v_{1}+\tilde{\theta}v_{0}\\
\pi\left(h_{s_{0}}\right)v_{0} & =\pi\left(h_{\tau}h_{s_{1}}\right)v_{1}=\theta v_{1}.
\end{align*}
 The relations show that $\varphi$ is a homomorphism of representations
of $H_{\phi}$. Therefore if $V$ is irreducible it is equal to the
image to $\varphi$, and therefore $V$ is isomorphic to a quotient
of $U_{\theta}$ (see the Preliminaries).
\end{proof}
\begin{cor}
For every $\theta\ne\pm1,\pm q$ the representation $\left(\pi_{\theta}^{\phi},U_{\theta}\right)$
is irreducible, and is isomorphic to the representation $\left(\pi_{q\theta^{-1}}^{\phi},U_{q\theta^{-1}}\right)$. 

There are four one dimensional representations which occur as quotients
of $\left(\pi_{\pm1}^{\phi},U_{\pm1}\right)$,$\left(\pi_{\pm q}^{\phi},U_{\pm q}\right)$: 

Two \emph{trivial representations} $\left(\pi_{T}^{\phi,\pm},U_{T}^{\pm}\right)$,
where $\tau$ acts by $\pm1$ and $h_{s_{0}},h_{s_{1}}$ act by $q$. 

Two \emph{Steinberg (}or\emph{ special) representations} $\left(\pi_{S}^{\phi,\pm},U_{S}^{\pm}\right)$,
where $\tau$ acts by $\pm1$ and $h_{s_{0}},h_{s_{1}}$ act by $-1$. 
\end{cor}

\begin{proof}
The only possible irreducible representations $\left(\pi,V\right)$
that are not $U_{\theta}$ are of dimension 1. In this case $\pi\left(h_{\tau}\right)$
acts by $\alpha_{\tau}=\pm1$, and $\pi\left(h_{s_{0}}\right)$ acts
by multiplication by a scalar $\alpha_{0}$. By the Iwahori-Hecke
relation $\left(h_{s_{0}}+1\right)\left(h_{s_{0}}-q\right)=0$, we
have $\alpha_{0}=-1$ or $\alpha_{0}=q$. Since $h_{s_{1}}=h_{\tau}h_{s_{0}}h_{\tau}$
the operator $\pi\left(h_{s_{1}}\right)$ also acts by $\alpha_{0}$.
On the other hand, each choice of $\alpha_{\tau}\in\left\{ \text{\ensuremath{\pm}1}\right\} $
and $\alpha_{0}\in\left\{ -1,q\right\} $ defines a one-dimensional
representation, as an easy verification of the Iwahori-Hecke relations
show. This gives us the four representations $U_{S}^{\pm},U_{S}^{\pm}$.
\end{proof}
Let us explain the connection between the $H_{0}$-representation
$\left(\pi_{\theta}^{0},V_{\theta}\right)$ and the $H_{\phi}$-representation
$\left(\pi_{\theta}^{\phi},U_{\theta}\right)$. Recall that we defined
in Definition \ref{def:Full_Hecke_Algebra} a larger algebra $H$
containing both $H_{0}$ and $H_{\phi}$ as subalgebras (with different
units). 

Now, $H\cdot Id_{H_{\phi}}$ is a right $H_{\phi}$-module and a left
$H$-representation, so (see the Preliminaries) given a representation
$\left(\pi,V\right)$ of $H_{\phi}$, $H\cdot Id_{H_{\phi}}\otimes_{H_{\phi}}V$
is an $H$-representation. To simplify notations we write it as $H\otimes_{H_{\phi}}V$,
which is well defined if we extend the tensor notation to modules
in which $Id_{H_{\phi}}$ does not act as the identity.
\begin{prop}
\label{prop:induction and restriction}We can induce an $H_{\phi}$-representation
$\left(\pi,U\right)$ to an $H$-representation $\left(\text{ind}_{H_{\phi}}^{H}\pi,\text{ind}_{H_{\phi}}^{H}U\right)$
by choosing $\text{ind}_{H_{\phi}}^{H}U=H\otimes_{H_{\phi}}U$. We
can restrict an $H$-representation $\left(\pi,W\right)$ to an $H_{\phi}$-representation
$\left(\text{res}_{H_{\phi}}^{H}\pi,\text{res}_{H_{\phi}}^{H}W\right)$
by choosing $\text{res}_{H_{\phi}}^{H}W=\pi\left(Id_{H_{\phi}}\right)W$.
Induction and restriction define a bijection between isomorphism classes
of irreducible finite dimensional $H_{\phi}$-representations and
$H$-representations.
\end{prop}

\begin{proof}
To simplify notations in the proof, we let $H$ and $H_{\phi}$ act
directly on the representation spaces without mentioning $\pi$.

The idea behind the proof is the decomposition 
\begin{equation}
H=H_{\phi}\oplus H_{\phi}u\oplus dH_{\phi}\oplus dH_{\phi}u,\label{eq:Decomposition}
\end{equation}
where $H_{\phi}u=\left\{ hu\colon h\in H_{\phi}\right\} $, and similarly
for $dH_{\phi}$, $dH_{\phi}u$. This decomposition follows immediately
from the relations stated in Proposition \ref{prop:H-speciication}.

We start by showing that $U$ and $\mbox{res}_{H_{\phi}}^{H}\mbox{ind}_{H_{\phi}}^{H}U$
are isomorphic $H_{\phi}$-representations, and $W$ and $\mbox{ind}_{H_{\phi}}^{H}\mbox{res}_{H_{\phi}}^{H}W$
are isomorphic $H$-representations. It proves that induction and
restriction define a bijection between isomorphism classes of $H_{\phi}$-representations
and $H$-representations. 

From Decomposition \ref{eq:Decomposition}, we have that $H\cdot Id_{H_{\phi}}=H_{\phi}\oplus dH_{\phi}$,
so 
\begin{align*}
\mbox{ind}_{H_{\phi}}^{H}U & =H\otimes_{H_{\phi}}U=\left\{ d\otimes v:v\in U\right\} \oplus\left\{ Id_{H_{\phi}}\otimes v:v\in U\right\} \\
 & =\left(d\otimes U\right)\oplus\left(Id_{H_{\phi}}\otimes U\right).
\end{align*}

On the first factor $Id_{H_{\phi}}$ acts by $0$, so 
\begin{align*}
\mbox{res}_{H_{\phi}}^{H}\mbox{ind}_{H_{\phi}}^{H}U & =Id_{H_{\phi}}\cdot H\otimes_{H_{\phi}}U=\left(Id_{H_{\phi}}d\otimes U\right)\oplus\left(Id_{H_{\phi}}\otimes U\right)\\
 & =Id_{H_{\phi}}\otimes U.
\end{align*}
It is immediate that $v\to Id_{H_{\phi}}\otimes v$ is an isomorphism
of the $H_{\phi}$-representations $U$ and $\mbox{res}_{H_{\phi}}^{H}\mbox{ind}_{H_{\phi}}^{H}U=Id_{H_{\phi}}\otimes U$.

Given an $H$-representation $W$, from $Id_{H}=Id_{H_{\phi}}+Id_{H_{0}}=Id_{H_{\phi}}+\left(q+1\right)^{-1}du$,
we have
\[
W=Id_{H}W=Id_{H_{\phi}}W\oplus Id_{H_{0}}W=Id_{H_{\phi}}W\oplus duW.
\]
We have $uW=Id_{H_{\phi}}W$, so $duW=dId_{H_{\phi}}W=dW$, so $W=Id_{H_{\phi}}W\oplus dW$.
Therefore,
\begin{align*}
\mbox{ind}_{H_{\phi}}^{H}\mbox{res}_{H_{\phi}}^{H}W & =H\otimes_{H_{\phi}}Id_{H_{\phi}}W=\\
 & =\left(Id_{H_{\phi}}\otimes_{H_{\phi}}Id_{H_{\phi}}W\right)\oplus\left(d\otimes_{H_{\phi}}Id_{H_{\phi}}W\right).
\end{align*}
The vector spaces $Id_{H_{\phi}}W$ and $Id_{H_{\phi}}\otimes_{H_{\phi}}Id_{H_{\phi}}W$
are naturally isomorphic as $H_{\phi}$-representations. Similarly
$dV=dId_{H_{\phi}}W$ and $d\otimes_{H_{\phi}}Id_{H_{\phi}}W$ are
naturally isomorphic as vector spaces. Therefore, $W$ and $\mbox{ind}_{H_{\phi}}^{H}\mbox{res}_{H_{\phi}}^{H}W$
are isomorphic as vector spaces. It is also easy to see that the actions
of $d,u$ and the elements of $H_{\phi}$ on the two spaces agree,
so by Proposition \ref{prop:H-speciication}, $W$ and $\mbox{ind}_{H_{\phi}}^{H}\mbox{res}_{H_{\phi}}^{H}W$
are isomorphic $H$-representations.

It is also obvious that if $W$ is a finite dimensional $H$-representation
then $\mbox{res}_{H_{\phi}}^{H}W$ is finite dimensional, and if $U$
is a finite dimensional $H_{\phi}$-representation then $\mbox{ind}_{H_{\phi}}^{H}U=\left(d\otimes U\right)\oplus\left(Id_{H_{\phi}}\otimes U\right)$
is finite dimensional. 

It also follows from our explicit description above that if $U_{1}\subset U$
is a proper non-trivial $H_{\phi}$-subrepresentation then $\mbox{ind}_{H_{\phi}}^{H}U_{1}\subset\mbox{ind}_{H_{\phi}}^{H}U$
is a proper non-trivial $H$-representation. Similarly, if $W_{1}\subset W$
is a proper non-trivial $H$-representation then $\mbox{res}_{H_{\phi}}^{H}W_{1}\subset\mbox{res}_{H_{\phi}}^{H}W$
is a proper non-trivial $H_{\phi}$-representation.

Finally, if $U$ is irreducible, then $\mbox{ind}_{H_{\phi}}^{H}U$
is irreducible, otherwise it has a proper subrepresentation, whose
restriction is a proper subrepresentation of $U$. Similarly, if $W$
is irreducible then $\mbox{res}_{H_{\phi}}^{H}W$ is irreducible.
It concludes the proof.
\end{proof}
A generalized version of Proposition \ref{prop:induction and restriction}
appears in \cite[Section 10]{kamber2016lpcomplex}. 

To make Proposition \ref{prop:induction and restriction} more explicit,
we extend the embedding $\Phi\colon H_{\phi}\to M_{2\times2}\left(\mathbb{C}\left[\theta,\theta^{-1}\right]\right)$
to an embedding $\Phi'\colon H\to M_{3\times3}\left(\mathbb{C}\left[\theta,\theta^{-1}\right]\right)$,
satisfying
\[
\Phi'(h_{w})=\left(\begin{array}{cc}
\Phi(h_{w}) & \begin{array}{c}
0\\
0
\end{array}\\
\begin{array}{cc}
0 & 0\end{array} & 0
\end{array}\right)
\]
for $w\in\hat{W}$, and 
\[
\Phi'\left(u\right)=\left(\begin{array}{ccc}
0 & 0 & 1\\
0 & 0 & \theta\\
0 & 0 & 0
\end{array}\right),\,\Phi'\left(d\right)=\left(\begin{array}{ccc}
0 & 0 & 0\\
0 & 0 & 0\\
1 & \tilde{\theta} & 0
\end{array}\right).
\]

Once again, this embedding can be derived from the presentation of
the algebra using generators and relations.

Notice that using this description, 
\[
\Phi'(A)=\Phi'\left(dh_{\tau}u\right)=\left(\begin{array}{ccc}
0 & 0 & 0\\
0 & 0 & 0\\
1 & \tilde{\theta} & 0
\end{array}\right)\left(\begin{array}{ccc}
0 & 1 & 0\\
1 & 0 & 0\\
0 & 0 & 0
\end{array}\right)\left(\begin{array}{ccc}
0 & 0 & 1\\
0 & 0 & \theta\\
0 & 0 & 0
\end{array}\right)=\left(\begin{array}{ccc}
0 & 0 & 0\\
0 & 0 & 0\\
0 & 0 & \theta+\tilde{\theta}
\end{array}\right).
\]

The reader may verify by calculating $\Phi'\left(A_{m}\right)=\Phi'\left(dh_{\tau}h_{NB}^{m-1}u\right)$
that one recovers the Satake isomorphism of Subsection \ref{subsec:Representations-of_H_0}.
\begin{prop}
We denote by $\left(\pi_{\theta},W_{\theta}\right),\theta\ne\pm1,\pm q$,
$\left(\pi_{T}^{\pm},W_{T}^{\pm}\right),\left(\pi_{S}^{\pm},W_{S}^{\pm}\right)$
the irreducible representations of $H$.

When restricted to an $H_{0}$-representation by $V=Id_{H_{0}}\cdot W$,
the corresponding representations are:

(1) For $\left(\pi_{\theta},W_{\theta}\right),\theta\ne\pm1,\pm q$:
$\left(\pi_{\theta}^{0},V_{\theta}\right)$.

(2) For $\left(\pi_{T}^{\pm},W_{T}^{\pm}\right)$: $\left(\pi_{\pm q}^{0},V_{\pm q}\right)=\left(\pi_{\pm1}^{0},V_{\pm1}\right)$.

(3) For $\left(\pi_{S}^{\pm},W_{S}^{\pm}\right)$: the $0$-representation.
\end{prop}

Note that we recovered all the irreducible $H_{0}$-representation.
\begin{rem}
The theory presented here is a very simple case of the general representation
theory of affine Iwahori-Hecke algebras. See \cite{macdonald1996affine}.

The names of the trivial representation and the Steinberg representation
come from corresponding representations of the automorphism group
of the tree.
\end{rem}

\subsubsection{Unitary Representations}

We want to identify the irreducible unitary representations of $H_{\phi}$.
It is immediate that the one dimensional representations are unitary
so we look at $\left(\pi_{\theta}^{\phi},U_{\theta}\right)$. Unitary
representations $\left(\pi,U\right)$ satisfy the following: 
\begin{enumerate}
\item The adjoint of $\pi\left(h_{NB}\right)=\pi\left(h_{\tau s_{0}}\right)$
is $\pi\left(h_{NB}^{*}\right)=\pi\left(h_{s_{0}\tau}\right)$, and
in unitary representation they have complex conjugate eigenvalues.
The eigenvalues of $h_{NB}=h_{\tau}h_{s_{0}}$ on $U_{\theta}$ are
$\theta$ and $q\theta^{-1}$. The eigenvalues of its adjoint $h_{NB}^{*}=h_{s_{0}\tau}$
are also $\theta$ and $q\theta^{-1}$. Therefore either $\theta=\bar{\theta}$,
i.e.~$\theta$ is real, or $\theta=q\bar{\theta}^{-1}$, i.e.~$\left|\theta\right|^{2}=q$.
\item The eigenvalues of $\pi\left(h_{NB}\right)$ are of absolute value
$\le q$, since $h_{NB}=h_{\tau}h_{s_{0}}$ and the eigenvalues and
therefore the norms of $\pi\left(h_{\tau}\right),\pi\left(h_{s_{0}}\right)$
are bounded by $1,q$. This condition bounds $\theta$ to $1\le\left|\theta\right|\le q$,
and since $U_{\pm1},U_{\pm q}$ are reducible the actual bound is
$1<\left|\theta\right|<q$. 
\end{enumerate}
Summarizing, we proved one direction of the following proposition. 
\begin{prop}
The unitary irreducible representations of $H_{\phi}$ are the following
representations:

1. $\left(\pi_{\theta}^{\phi},U_{\theta}\right)$, for $\left|\theta\right|=q^{\nicefrac{1}{2}}$.

2. $\left(\pi_{\theta}^{\phi},U_{\theta}\right)$, for $\theta$ real
$1<\left|\theta\right|<q$.

3. The one dimensional representations: $\left(\pi_{T}^{\phi,\pm},U_{T}^{\pm}\right)$,
$\left(\pi_{S}^{\phi,\pm},U_{S}^{\pm}\right)$.
\end{prop}

The proposition says that the algebraic definition of a unitary $H_{\phi}$-representation
capture the combinatoric bounds we found in Subsection \ref{subsec:Action-on-Finite-graphs}
on the Satake parameter. This is a very simple case of a general result
of Barbasch and Moy (\cite{barbasch1993reduction}). To complete the
proof we need to prove that $\left(\pi_{\theta}^{\phi},U_{\theta}\right)$
for $\theta$ as in the proposition is indeed unitary. Since we will
not use this part and it is slightly technical, we skip it and refer
the reader to \cite[Section 9]{savin2002lectures}, where a similar
claim is proven in the context of the representation theory of $p$-adic
groups.

\subsection{\label{subsec:Geometric-Realization-1}Geometric Realization and
the $L^{p}$-Spectrum}

Similarly to the vertex Hecke algebra, representations of the Iwahori-Hecke
algebra can be realized on the tree. We will only describe the spherical
model, although there also exists a sectorial model of an irreducible
representation.

Recall that by Lemma \ref{lem:Tree_is_building} and Definition \ref{lem:Tree_is_building}
we have a distance $d\colon E_{T}\times E_{T}\rightarrow\hat{W}$
and that for a given $e_{0}\in E_{T}$ the number of $e\in E_{T}$
with $d(e_{0},e)=w$ is $q^{l(w)}$.
\begin{defn}
Given a representation $\left(\pi,U\right)$ of $H_{\phi}$, $u\in U$
and $\varphi\in U^{\ast}$ we call the function $c_{\varphi,u}\colon H_{\phi}\rightarrow\mathbb{C}$,
$c_{\varphi,u}(h)=\left\langle \varphi,\pi\left(h\right)u\right\rangle $
a \emph{matrix coefficient} of $U$. 

Fix $e_{0}\in E_{T}$. For every matrix coefficient we associate a
\emph{geometric realization} $f_{\varphi,u}^{e_{0}}=f_{\varphi,u}\in\mathbb{C}^{V_{T}}$
given by $f_{\varphi,u}(e)=\frac{1}{q^{l(d(e_{0},e))}}\left\langle \varphi,\pi\left(h_{d(e_{0},e)}\right)u\right\rangle $.

We say that $\left(\pi,U\right)$ is \emph{$p$-finite }if for every
$u\in U$, $\varphi\in U^{\ast}$, $f_{\varphi,u}(v)\in L^{p}(E_{T})$,
or equivalently 
\[
\sum_{w\in\hat{W}}q^{l(w)(1-p)}\left|\left\langle \varphi,\pi\left(h_{w}\right)u\right\rangle \right|^{p}<\infty.
\]

We say that $\left(\pi,U\right)$ is $p$-tempered if it is $p'$-finite
for every $p'>p$.
\end{defn}

A nice feature of working with the Iwahori-Hecke algebra is that $p$-temperedness
is directly related to the eigenvalues of $h_{NB}$. 
\begin{prop}
\label{prop:non-backtracking equiv}Let $\left(\pi,U\right)$ be a
finite dimensional representation of $H_{\phi}$, and let $\rho_{U}(h_{NB})$
be the largest absolute value of an eigenvalue of $\pi\left(h_{NB}\right)$.
Then $\left(\pi,U\right)$ is $p$-tempered if and only if $\rho_{U}(h_{NB})\le q^{\nicefrac{\left(p-1\right)}{p}}$.

Therefore, $\left(\pi_{\theta}^{\phi},U_{\theta}\right)$ is $p$-tempered
if and only if $\max\{\left|\theta\right|,q\left|\theta\right|^{-1}\}\le q^{\nicefrac{\left(p-1\right)}{p}}$,
$\left(\pi_{S}^{\phi,\pm},U_{S}^{\pm}\right)$ is $1$-tempered and
$\left(\pi_{T}^{\phi,\pm},U_{T}^{\pm}\right)$ is not $p$-finite
for any $p<\infty$.
\end{prop}

\begin{proof}
Every $w\in W$ can be written uniquely as $w=\tau^{\delta_{\tau}}w_{NB}^{m}s_{1}^{\delta_{1}}$,
for $\delta_{\tau},\delta_{1}\in\{0,1\}$ and $m\ge0$. Thus, $p$-finiteness
is equivalent to

\[
\sum_{\delta_{\tau},\delta_{1}\in\{0,1\}}q^{\delta_{1}(1-p)}\sum_{m\ge0}q^{(1-p)m}\left|\left\langle h_{\tau}^{\delta_{\tau}}\varphi,\pi\left(h_{NB}^{m}\right)h_{s_{1}}^{\delta_{1}}u\right\rangle \right|^{p}<\infty
\]
for every $u\in U$, $\varphi\in U^{\ast}$.

This is reduced to the convergence of $\sum_{m\ge0}q^{(1-p)m}\left|\left\langle \varphi,\pi\left(h_{NB}\right)^{m}u\right\rangle \right|^{p}$
for every $u\in U$, $\varphi\in U^{\ast}$. If $u$ is an eigenvector
of $\pi\left(h_{NB}\right)$ with eigenvalue $\theta$ with $\left|\theta\right|\ge q^{\nicefrac{\left(p-1\right)}{p}}$
and $\left\langle \varphi,u\right\rangle \ne0$, then the series diverges.
For the other direction, the theory of matrix norms says that for
every $u\in U$ and $\varphi\in U^{*}$, $\limsup_{m}\left|\left\langle \varphi,\pi\left(h_{NB}\right)^{m}u\right\rangle \right|^{1/m}\le\rho_{U}(h_{NB})$,
which shows that if $\rho_{U}(h_{NB})<q^{\nicefrac{\left(p-1\right)}{p}}$
the series converges.
\end{proof}
As with representations of $H_{0}$, geometric realizations allow
us to consider every irreducible representation of $H_{\phi}$ as
a subrepresentation of $\mathbb{\mathbb{C}}^{E_{T}}$. 

The definition can be extended to $H$ and agrees with the corresponding
definition of $p$-temperedness of $H_{0}$-representations for irreducible
$H$-representations whose restriction to $H_{0}$ is non-zero.

The arguments of Subsection \ref{subsec:L_p-Spectrum-of-Hecke} can
be extended to the following theorem:
\begin{thm}
\label{thm:Main-Theorem for graphs}Let $V$ be a finite dimensional
representation of $H$. Let $p\ge2$. The following are equivalent:

1. The eigenvalues of every $h\in H$ are contained in the spectrum
of $h$ on $L^{p}(V_{T}\oplus E_{T})$. 

2. $V$ is $p$-tempered.

Specifically, the eigenvalues of $h_{NB}$ on $L^{p}(E_{T})$ are
$\{\pm1\}\cup\left\{ \theta\in\mathbb{C}\backslash\{0\}:\max\{\left|\theta\right|,q\left|\theta\right|^{-1}\}<q^{\nicefrac{\left(p-1\right)}{p}}\right\} $
and the approximate point spectrum of $h_{NB}$ on $L^{p}(E_{T})$
is $\{\pm1\}\cup\left\{ \theta\in\mathbb{C}\backslash\{0\}:\max\{\left|\theta\right|,q\left|\theta\right|^{-1}\}\le q^{\nicefrac{\left(p-1\right)}{p}}\right\} $.
\end{thm}

\subsection{\label{subsec:L_p-Expander-Theorem}The $L^{p}$-Expander Theorem}

We summarize the discussion above by the following theorem. For simplicity,
we look at a finite non-bipartite graph $X$, and denote $L_{0}^{2}(V_{X})=\left\{ f\in L^{2}(V_{X}):\sum_{v\in V_{X}}f(v)=0\right\} $
and $L_{0}^{2}(E_{X})=\left\{ f\in L^{2}(E_{X}):\sum_{e\in E_{X}}f(e)=0\right\} $.
\begin{thm}
\label{thm:Full Theorem}Let $X$ be a finite, connected, non-bipartite,
$\left(q+1\right)$-regular graph. For $p\ge2$, the following are
equivalent:

1. Every eigenvalue $\lambda$ of $A_{X}$ on $L_{0}^{2}(V_{X})$
satisfies $\left|\lambda\right|\le q^{\nicefrac{1}{p}}+q^{\nicefrac{\left(p-1\right)}{p}}$.

2. The only representations appearing in the decomposition of the
$H_{0}$-action on $L^{2}(V_{X})$ are $\left(\pi_{\theta}^{0},V_{\theta}\right)$
with $\max\{\left|\theta\right|,q\left|\theta\right|^{-1}\}\le q^{\nicefrac{\left(p-1\right)}{p}}$
and $\left(\pi_{q}^{0},V_{q}\right)$.

3. The only representations appearing in the decomposition of the
$H_{\phi}$-action on $L^{2}(E_{X})$ are $\left(\pi_{\theta}^{\phi},U_{\theta}\right)$
with $\max\{\left|\theta\right|,q\left|\theta\right|^{-1}\}\le q^{\nicefrac{\left(p-1\right)}{p}}$,
$\left(\pi_{S}^{\phi,\pm},U_{S}^{\pm}\right)$ and $\left(\pi_{T}^{\phi,\pm},U_{T}\right)$.

4. The $H_{0}$-representation $\left(\pi_{X},L_{0}^{2}(V_{X})\right)$
is $p$-tempered.

5. The eigenvalues of $\pi_{X}\left(h\right)$ for every $h\in H$
on $L_{0}^{2}(V_{X})\oplus L_{0}^{2}(E_{X})$ are contained in the
spectrum of $h$ on $L^{p}(V_{T}\oplus E_{T})$. 

6. For every $k$, every eigenvalue $\lambda_{k}$ of $\pi_{X}\left(A_{k}\right)$
on $L_{0}^{2}(V_{X})$ satisfies $\left|\lambda_{k}\right|\le A_{k}(q^{\nicefrac{\left(p-1\right)}{p}})\le(k+1)q^{\nicefrac{k\left(p-1\right)}{p}}$.

7. Every eigenvalue $\lambda$ of Hashimoto's non-backtracking operator
$\pi_{X}\left(h_{NB}\right)$ on $L_{0}^{2}(E_{X})$ satisfies $\left|\lambda\right|\le q^{\nicefrac{\left(p-1\right)}{p}}$.
\end{thm}

\subsection{\label{subsec:Bipartite-Biregular-Graphs}Bipartite Biregular Graphs}

In this subsection we show how to extend the previous results to biregular
graphs.

Let $\tilde{T}$ be a biregular tree, i.e.~each vertex $v$ is colored
by $t(v)\in\left\{ 0,1\right\} $, each edge contains one vertex of
type $0$ and one vertex of type $1$, and each vertex of type $i\in\left\{ 0,1\right\} $
is contained in $q_{i}+1$ edges. As the case $q_{0}=q_{1}$ was covered
above, we assume $q_{1}>q_{0}\ge1$.

Following Proposition \ref{prop:Hecke_commutes} and Proposition \ref{prop:Hecke_commutes-1},
we define:
\begin{defn}
The vertex Hecke algebra $\tilde{H}_{0}$ is the algebra of row and
column finite operators acting on $\mathbb{C}^{V_{\tilde{T}}}$ and
commuting with automorphisms of $\tilde{T}$. 
\end{defn}

Since we have two types of vertices and an automorphism never sends
a vertex of one type to the other (since $q_{1}>q_{0})$, the algebra
$\tilde{H}_{0}$ is slightly more complicated than the algebra $H_{0}$
of the regular case. In particular, $\tilde{H}_{0}$ is not generated
by the adjacency operator $A$ and is not commutative. However, the
algebra still contains the operators $A_{m}\colon\mathbb{C}^{V_{\tilde{T}}}\to\mathbb{C}^{V_{\tilde{T}}}$.
Each operator $A_{m}$ sums a function on a sphere of radios $m$
on $\tilde{T}$ around each vertex, which is of approximate size $\left(\sqrt{q_{0}q_{q}}\right)^{m}$
(the exact size depends on the type of the vertex).

As for operators acting on edges, since each edge has a natural ``direction''
from vertices of type 0 to vertices of type 1, it is simpler to consider
the set $\tilde{E}_{\tilde{T}}$ of non-directed edges. Give each
non-directed edge $e=\{x,y\}\in\tilde{E}_{\tilde{T}}$ a direction
from its 0-vertex $x$ to its 1-vertex $y$. We therefore write $e=\left(x,y\right)$.

Biregular trees also have an Iwahori-Hecke algebra attached to them,
which was used by Hashimoto in \cite{hashimoto1989zeta}. This time
it is natural to use the standard Iwahori-Hecke algebra and not the
extended one (see Remark \ref{rem:history remark}). The shortest
way to define this algebra is as follows:
\begin{defn}
The Iwahori-Hecke (or edge-Hecke) algebra $\tilde{H}_{\phi}$ is the
algebra of row and column finite operators acting on $\mathbb{C}^{\tilde{E}_{\tilde{T}}}$
and commuting with automorphisms of $\tilde{T}$.
\end{defn}

The description of $\tilde{H}_{\phi}$ is very similar to the regular
case. In particular, we have the operators $h_{s_{0}}$ and $h_{s_{1}}$
(but not $h_{\tau}$), defined as in Subsection \ref{subsec:The-Directed-Edge-Hecke}.
The basis of $\tilde{H}_{\phi}$ consists of the operators $h_{w}$,
$w\in W$, where $W$ is the (non-extended) Coxeter group $W=\left\langle s_{0},s_{1}:s_{0}^{2}=s_{1}^{2}=1\right\rangle $
(see Remark \ref{rem:history remark}). The Iwahori-Hecke relations
are:
\begin{eqnarray*}
h_{s_{0}}^{2} & = & (q_{0}-1)h_{s_{0}}+q_{0}Id\\
h_{s_{1}}^{2} & = & (q_{1}-1)h_{s_{1}}+q_{1}Id\\
h_{w}h_{s} & = & h_{ws}\,\,\,\,\,\,\,\,\,\,\,\,\,\,\,\,\,\,\,\,\mbox{if}\,\,l(ws)>l(w).
\end{eqnarray*}

Hashimoto's non-backtracking operator $h_{NB}$ is not part of our
algebra. However, we do have the non-backtracking operator $\tilde{h}_{NB}=h_{s_{1}s_{0}}=h_{s_{1}}h_{s_{0}}$
(which corresponds to $h_{NB}^{2}$ in the regular graph case).
\begin{defn}
The full graph Hecke algebra $\tilde{H}$ is the algebra of row and
column finite operators acting on $\mathbb{C}^{\tilde{E}_{\tilde{T}}}\oplus C^{V_{\tilde{T}}}$
and commuting with tree automorphisms.
\end{defn}

In this case it is useful to define the raising operators $u_{0},u_{1}\colon C^{V_{\tilde{T}}}\to\mathbb{C}^{\tilde{E}_{\tilde{T}}}$,
and the lowering operators $d_{0},d_{1}\colon\mathbb{C}^{\tilde{E}_{\tilde{T}}}\to C^{V_{\tilde{T}}}$
by:
\[
\begin{array}{c}
u_{0}f(x,y)=f(x),\,u_{1}(x,y)=f(y)\\
d_{0}f(x)=\begin{cases}
\sum_{y\sim x}f(x,y) & x\text{ of type 0}\\
0 & \text{otherwise }
\end{cases},\,d_{1}f(y)=\begin{cases}
\sum_{x\sim y}f(x,y) & y\text{ of type 1}\\
0 & \text{otherwise }
\end{cases}
\end{array}
\]

We extend $u_{0},u_{1},d_{0},d_{1}$ as usual to operators in $\tilde{H}$.
The relations satisfied are:
\begin{align*}
u_{0}d_{0} & =h_{s_{0}}+Id_{\tilde{H}_{\phi}},\,u_{1}d_{1}=h_{s_{1}}+Id_{\tilde{H}_{\phi}}\\
A & =d_{0}u_{1}+d_{1}u_{0}.
\end{align*}

The algebra $\tilde{H}$ and the raising and lowering operators allow
us to transfer results between $\tilde{H}_{0}$ and $\tilde{H}_{\phi}$.
In addition, The restriction operator $\tilde{W}\rightarrow\tilde{U}=Id_{\tilde{H}_{\phi}}\tilde{W}$
from an $\tilde{H}$-representation $\tilde{W}$ to an $\tilde{H}_{\phi}$-representation
$\tilde{U}$ defines a bijection between equivalence classes of irreducible
representations of $\tilde{H}$ and $\tilde{H}_{\phi}$, as in Proposition
\ref{prop:induction and restriction}. 

The $L^{p}$-theory remains essentially the same. It is summarized
in the following theorem:
\begin{thm}
\label{thm:Bipartite-Lp-theorem}Let $\left(\pi,\tilde{W}\right)$
be a finite dimensional representation of $\tilde{H}$. Let $p\ge2$.
The following are equivalent:

1. Each eigenvalue $\lambda$ of $\pi\left(h\right)$ for every $h\in\tilde{H}$
on $\tilde{W}$ is contained in the approximate point spectrum of
$h$ on $L^{p}(V_{\tilde{T}}\oplus\tilde{E}_{\tilde{T}})$. 

2. The representation $\left(\pi,\tilde{W}\right)$ is $p$-tempered,
i.e.~all of its geometric realizations are in $L^{p+\epsilon}(V_{\tilde{T}}\oplus\tilde{E}_{\tilde{T}})$
for every $\epsilon>0$.

3. Each eigenvalue $\theta'$ of $\pi\left(\tilde{h}_{NB}\right)$
on $Id_{\tilde{H}_{\phi}}\tilde{W}$ satisfies $\left|\theta'\right|\le\left|q_{0}q_{1}\right|^{\nicefrac{\left(p-1\right)}{p}}$.
\end{thm}

Since this theorem is a special case of the generalized theory in
\cite{kamber2016lpcomplex} (specifically, Theorem 1.6 and Corollary
1.12), we only give its sketch. The fact that (2) derives (1) is as
in Proposition \ref{prop:Temperd derives wea containment}. The equivalence
between (2) and (3) is as in Proposition \ref{prop:non-backtracking equiv}.
Finally, to prove that (1) derives (3) one generalizes Proposition
\ref{prop:Kesten's}. It also gives a more qualitative part, given
as follows:
\begin{prop}
\label{prop:Am-norm-Bipartite}The norm of $A_{m}\in\tilde{H}_{0}$
on $L^{p}(V_{\tilde{T}})$ (and therefore on every $p$-tempered unitary
representation) is bounded by $(m+1)q_{1}\left(q_{0}q_{1}\right)^{(\nicefrac{m}{2})\nicefrac{\left(p-1\right)}{p}}$.
\end{prop}

The proposition states that up to $O(m)$, $A_{m}$ is bounded by
the $\nicefrac{\left(p-1\right)}{p}$-th power of the number of vertices
it sums.
\begin{defn}
A $\left(q_{0}+1,q_{1}+1\right)$-biregular graph $X$ is \emph{an
$L^{p}$-expander} if every non-trivial $\tilde{H}$-subrepresentation
of $\left(\pi_{X},L^{2}\left(\tilde{E}_{X}\oplus V_{X}\right)\right)$
is $p$-tempered.
\end{defn}

The theorems above are general and do not require the classification
of $\tilde{H}$-representations. However, understanding the exact
connection between the eigenvalues of the adjacency operator $A$
and $p$-temperedness does require the classification. To classify
irreducible representations, we embed $\tilde{H}$ in $M_{4\times4}\left(\mathbb{C}\left[\theta^{\prime},\theta^{\prime-1}\right]\right)$
($\theta^{\prime}$ indeterminate) by%
\begin{align*}
h_{s_{0}} & \rightarrow\left(\begin{array}{cccc}
0 & q_{0} & 0 & 0\\
1 & q_{0}-1 & 0 & 0\\
0 & 0 & 0 & 0\\
0 & 0 & 0 & 0
\end{array}\right),\,h_{s_{1}}\to\left(\begin{array}{cccc}
q_{1}-1 & \theta^{\prime} & 0 & 0\\
\theta^{\prime-1}q_{1} & 0 & 0 & 0\\
0 & 0 & 0 & 0\\
0 & 0 & 0 & 0
\end{array}\right)\\
\tilde{h}_{NB} & \rightarrow\left(\begin{array}{cccc}
\theta^{\prime-1} & q_{0}\left(q_{1}-1\right)+\theta^{\prime}\left(q_{0}-1\right) & 0 & 0\\
0 & \theta^{\prime-1}q_{0}q_{1} & 0 & 0\\
0 & 0 & 0 & 0\\
0 & 0 & 0 & 0
\end{array}\right)\\
D_{0} & \to\left(\begin{array}{cccc}
0 & 0 & 0 & 0\\
0 & 0 & 0 & 0\\
1 & q_{0} & 0 & 0\\
0 & 0 & 0 & 0
\end{array}\right),\,U_{0}\to\left(\begin{array}{cccc}
0 & 0 & 1 & 0\\
0 & 0 & 1 & 0\\
0 & 0 & 0 & 0\\
0 & 0 & 0 & 0
\end{array}\right)\\
D_{1} & \to\left(\begin{array}{cccc}
0 & 0 & 0 & 0\\
0 & 0 & 0 & 0\\
0 & 0 & 0 & 0\\
q_{1} & \theta^{\prime} & 0 & 0
\end{array}\right),\,U_{1}\to\left(\begin{array}{cccc}
0 & 0 & 0 & 1\\
0 & 0 & 0 & \theta^{\prime-1}\\
0 & 0 & 0 & 0\\
0 & 0 & 0 & 0
\end{array}\right).
\end{align*}

For $0\ne\theta'\in\mathbb{C}$, we denote the resulting 4-dimensional
$\tilde{H}$-representation by $\left(\tilde{\pi}_{\theta'},\tilde{W}_{\theta^{\prime}}\right)$.
Notice that $\theta^{\prime}$ corresponds to an eigenvalue of $\tilde{h}_{NB}$,
so $\tilde{W}_{\theta^{\prime}}$ will be similar to the representation
$W_{\theta^{\prime\nicefrac{1}{2}}}$ of Subsection \ref{subsec:Representation-Theory-of}.

The proposition below gives the classification of unitary $\tilde{H}$-representations
which can be found similarly to Subsection \ref{subsec:Representation-Theory-of}.
For simplicity we omit $\pi$ from the notations. The edge dimension
of a representation $\tilde{W}$ is the dimension of $Id_{\tilde{H}_{\phi}}\tilde{W}$.
The vertex dimension is the dimension of $Id_{\tilde{H}_{0}}\tilde{W}$.
The eigenvalues of $\tilde{h}_{NB}$ are calculated on $Id_{\tilde{H}_{\phi}}\tilde{W}$
and the eigenvalues of $A$ are calculated on $Id_{\tilde{H}_{0}}\tilde{W}$.
A similar classification can be found in \cite{hashimoto1989zeta}.
\begin{prop}
\label{prop:Classification of tree unitary}The unitary representations
of $\tilde{H}$ are the following representations:

1. The representation $\tilde{W}_{\theta^{\prime}}$, for (a) $\left|\theta^{\prime}\right|=\sqrt{q_{0}q_{1}}$
or (b) $\theta^{\prime}$ real with \textup{$1<\theta^{\prime}<q_{0}q_{1}$}
or (c) $\theta^{\prime}$ real with $-q_{1}<\theta^{\prime}<-q_{0}$.
This representation is of dimension 4: vertex dimension 2 and edge
dimension 2. 

The eigenvalues of $\tilde{h}_{NB}$ are $\theta^{\prime},q_{0}q_{1}\theta^{\prime-1}$
and it is $p$-tempered if and only\textbf{ }if $\max\left\{ \left|\theta^{\prime}\right|,q_{0}q_{1}\left|\theta^{\prime}\right|^{-1}\right\} \le\left(q_{0}q_{1}\right)^{\nicefrac{\left(p-1\right)}{p}}$. 

Write $\theta^{\prime}=\sqrt{q_{0}q_{1}}\hat{\theta}$. The eigenvalues
of $A$ are
\[
\pm\sqrt{\left(1+\theta^{\prime-1}q_{0}\right)\left(q_{1}+\theta^{\prime}\right)}=\pm\sqrt{\left(\hat{\theta}^{\nicefrac{1}{2}}\sqrt{q_{1}}+\hat{\theta}^{-\nicefrac{1}{2}}\sqrt{q_{0}}\right)\left(\hat{\theta}^{\nicefrac{1}{2}}\sqrt{q_{0}}+\hat{\theta}^{-\nicefrac{1}{2}}\sqrt{q_{1}}\right)}.
\]

(a) If $\left|\theta^{\prime}\right|=\sqrt{q_{0}q_{1}}$ the representation
is 2-tempered and the eigenvalues $\lambda_{\pm}$ of $A$ are $\lambda_{\pm}=\pm\left|\hat{\theta}^{\nicefrac{1}{2}}\sqrt{q_{1}}+\hat{\theta}^{-\nicefrac{1}{2}}\sqrt{q_{0}}\right|$,
and it holds that $\sqrt{q_{1}}-\sqrt{q_{0}}\le\left|\lambda_{\pm}\right|\le\sqrt{q_{1}}+\sqrt{q_{0}}$. 

(b)+(c) If $\left|\theta^{\prime}\right|\ne\sqrt{q_{0}q_{1}}$ the
representation is not $2$-tempered. (b) For $1<\theta^{\prime}<q_{0}q_{1}$
the eigenvalues $\lambda_{\pm}$ of $A$ satisfy $\sqrt{q_{1}}+\sqrt{q_{0}}<\left|\lambda\right|<\sqrt{\left(1+q_{0}\right)\left(1+q_{1}\right)}$.
(c) For $-q_{1}<\theta^{\prime}<-q_{0}$ the eigenvalues $\lambda_{\pm}$
of $A$ satisfy $0<\left|\lambda\right|<\sqrt{q_{1}}-\sqrt{q_{0}}$.

2. \emph{The Steinberg representation} $\tilde{W_{S}}$: $h_{0},h_{1}$
act by $-1$. This representation is of dimension 1: vertex dimension
0 and edge dimension 1. The eigenvalue of $\tilde{h}_{NB}$ is $1$.
There are no eigenvalues for $A$ since the vertex dimension is 0.
The representation is $1$-tempered

3. \emph{The trivial representation} $\tilde{W}_{T}$: $h_{0},h_{1}$
act by $q$. This representation is of dimension 3: vertex dimension
2 and edge dimension 1. The eigenvalue of $\tilde{h}_{NB}$ is $q_{0}q_{1}$.
The eigenvalues of $A$ are $\pm\sqrt{\left(1+q_{0}\right)\left(1+q_{1}\right)}$.
The representation is $\infty$-tempered.

4. The representations $\tilde{W}^{0}$: $h_{0}$ acts by $q$, $h_{1}$
acts by $-1$. This representation is of total dimension 2: vertex
dimension 1 and edge dimension 1. The eigenvalue of $\tilde{h}_{NB}$
is $-q_{0}$. The eigenvalue of $A$ is 0, with an eigenvector supported
on vertices of type 0. The representation is 2-finite.

5. The representations $\tilde{W}^{1}$: $h_{0}$ acts by $-1$, $h_{1}$
acts by $q_{1}$. This representation is of total dimension 2: vertex
dimension 1 and edge dimension 1. The eigenvalue of $\tilde{h}_{NB}$
is $-q_{1}$. The eigenvalue of $A$ is 0, with an eigenvector supported
on vertices of type 1. The representation is not $2$-tempered.
\end{prop}

Let us apply the classification to graphs, i.e.~to the decomposition
of $L^{2}\left(\tilde{E}_{X}\oplus V_{X}\right)$ as a unitary $\tilde{H}$-representation.

The trivial representation $\tilde{W}_{T}$ appears once: it is the
subrepresentation consisting of functions having constant value on
every type of face. By a dimension argument, the Steinberg representation
appears $\chi(X)+1=\left|\tilde{E}_{X}\right|-\left|V_{X}\right|+1$
times in the decomposition. The rest of the representations are either
$\tilde{W}_{\theta^{\prime}}$ or $\tilde{W}^{0}$ or $\tilde{W}^{1}$.
Counting dimensions again, we know that the difference between the
number of appearances of $\tilde{W}^{0}$ and $\tilde{W}^{1}$ is
$\left|V_{X}^{0}\right|-\left|V_{X}^{1}\right|$.

To make a graph an $L^{2}$-expander, we need that:
\begin{enumerate}
\item For each $\tilde{W}_{\theta^{\prime}}$ appearing in the decomposition,
$\theta^{\prime}$ will satisfy $\left|\theta^{\prime}\right|=\sqrt{q_{0}q_{1}}$.
\item The representation $\tilde{W}^{1}$ will not appear in the decomposition. 
\end{enumerate}
The classification prove the final part of Theorem \ref{thm:Bipartite theorem}:
\begin{thm}
A bipartite $\left(q_{0}+1,q_{1}+1\right)$-biregular graph is an
$L^{2}$-expander (i.e.~Ramanujan) if and only if the eigenvalues
of $A$ are $\lambda=0,\pm\sqrt{\left(1+q_{0}\right)\left(1+q_{1}\right)}$
or satisfy $\sqrt{q_{1}}-\sqrt{q_{0}}\le\left|\lambda\right|\le\sqrt{q_{1}}+\sqrt{q_{0}}$,
and the multiplicity of $\lambda=0$ is $\left|V_{X}^{0}\right|-\left|V_{X}^{1}\right|$.
\end{thm}

\section*{Acknowledgments }

This work was done as part of an M.Sc.~thesis submitted to the Hebrew
University of Jerusalem. The author would like to thank his adviser
Prof.~Alex Lubotzky for his guidance, support and patience regarding
the author's English, and the referee for multiple corrections and
improvements to the text.

\bibliographystyle{amsalpha}
\bibliography{database}

\end{document}